\documentclass[preprint,12pt]{elsarticle}
\usepackage[utf8]{inputenc}
\usepackage{lineno,hyperref}
\modulolinenumbers[5]
\usepackage{graphicx}
\usepackage{pst-all}
\usepackage{pstricks-add}
\usepackage[crop=off]{auto-pst-pdf}

\usepackage{pst-pdf}
\usepackage[T1]{fontenc}
\usepackage{amsmath,amssymb,latexsym}
\usepackage{amsthm}
\usepackage[center,small,sf]{caption}
\usepackage{mathtools}
\usepackage{diagbox}

\newtheorem{defin}{Definition}
\newtheorem{teor}{Theorem}
\newtheorem{prop}{Proposition}
\newtheorem{corol}{Corollary}

\renewcommand{\equiv}{\cong}
\usepackage[english]{babel}

\makeatletter
\def\ps@pprintTitle{%
   \let\@oddhead\@empty
   \let\@evenhead\@empty
   \let\@oddfoot\@empty
   \let\@evenfoot\@oddfoot
}
\makeatother
\begin{document}

\begin{frontmatter}

\title{
\textbf{Prime numbers. An alternative study using ova-angular rotations ($I$)}\\
}

\author{Y. Acevedo \corref{cor1}\fnref{label1}}
\fntext[label1]{\textit{email:} yaceved2@eafit.edu.co, Universidad EAFIT, Colombia, ORCID: 0000-0002-1640-9084.}

\begin{abstract}
Ova-angular rotations of a prime number are characterized, constructed u\-sing the Dirichlet theorem. The geometric properties arising from this theory are analyzed and some applications are presented, including a different way of studying Mersenne's primes, a proof of the existence of infinite primes of the form $\rho = k^2+1$ and the convergence of the sum of the inverses of the Mersenne primes. In this paper, the study ends by introducing the ova-angular square matrix.
\\
\end{abstract}

\begin{keyword}
Prime number\sep Ova-angular rotation\sep geometric properties \sep Dirichlet's theorem.
\end{keyword}

\end{frontmatter}
\footnotesize \textit{There is a normal factorization for a $ 5-smooth$ number whose existence is beautiful! The bases are the first three prime numbers and the exponents are the first natural numbers, observe its beauty: $2^{3}$. $3^{2}$. $5^{1}= 360$} (Autor).

\clearpage

 \pagestyle{myheadings}
 \pagenumbering{arabic}

\section{\textbf{Introduction}}

From functional analysis and number theory, direct contributions towards the study of prime numbers have been provided \cite{Vatwani2017, Murty2017, Avigad2014, Chen2013, Tao2013, Matomaki2017}. These contributions have linked the use of multiplicative functions, bounded sets, infimum, and supremum, as well as the importance of the norm and the modular congruences to establish differences or limit distances between consecutive primes, obtaining interesting results such as Zhang, who established an important dimension for the distance of consecutive primes \cite{Tianshu2013, Breitzman1970}.\\

It is worth noting that Dirichlet ($1823$), a Gauss's successor at the University of Göttingen, long investigated the prime numbers, linking the analysis and the theory of congruences on the distribution of these. Thus, addressing integer modules, quantities of prime numbers, and arithmetic progressions, he establishes his theorem "\textit{About primes in arithmetic progressions}" \cite{book:KennethH.Rosen}. With his theorem, Dirichlet presents the possibility of cla\-ssifying the prime numbers according to their residue modulo $n$, with $n$ integer. \\

In this work, a further interpretation of this theorem is developed and it is established that among all the possible integer modules that can be established to generate progressions with infinite prime numbers, the one that presents more stability is the $360$ module. Exactly this development is induced at a theoretical level under the denomination \emph{Ova-angular Rotations}. \\

It is estimated that with the study method presented below, the idea that primes in a similar way to energy-light (a concept of Physics), are quantized in rotational packages and their behavior is not far from physical interpretations, can be supported, being this one the following: \\

\emph {The prime numbers are distributed in ova-angular packages of $360^{\circ}$ of rotation, in other
words, are quantized. Establishing a somewhat abrupt comparison, the rotating disk would be an $<$atom$>$ with a nucleus at its center and its somewhat restless $<$electrons$>$ would be compared with the residues $\daleth_{\rho}$, which in interaction with the natural numbers in the direction axis would receive a kind of $<$energy$>$ capable enough for them to radiate a $<$photon$>$ comparable to a prime number.} \\

Adding the elements that provide ova-angular rotations and their applications within the field of Mathematics could bring closer to reality the idea that both Physics and Mathematics are closely related and there is a constitutional relationship between these knowledge sciences \cite{book:AcevedoY, book:SergeyevY}.

\section{
Introduction to ova-angular rotations of a prime number}

Consider for convenience a radius $r$, such that $r=\frac{180}{\pi} \approx 57.2958$ units and consider the circle $x^2+\left (y-\frac{180}{\pi} \right) ^2=r^2$. It is clear that this circle has a perimeter $L=2{\pi}r=360\,units.$

Take as a convention, for the beginning of this analysis, that such circle is located on the number line at the point $(0,0)$ as shown in Figure \eqref{Figura: $1$}.
\begin{figure}[ht]
\begin{center}
\includegraphics[width=9cm, height=6.5cm]{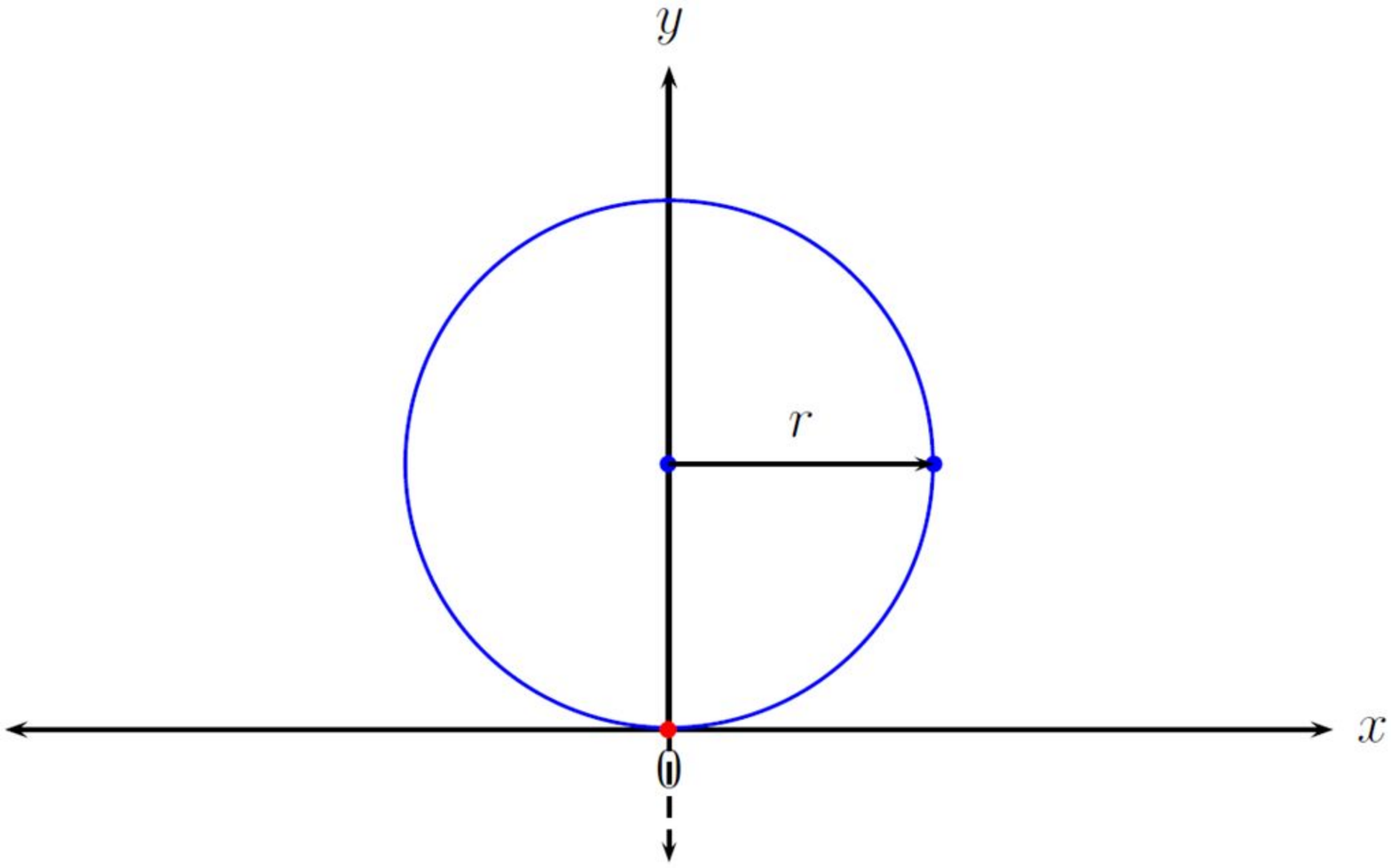}
\caption{Ova-angular rotation circle.}\label{Figura: $1$}
\end{center}
\end{figure}
Finally, consider that the circle has a constant speed $v= \frac{1\,\textit{unit}}{second}$ towards the positive axis $+x$.\\

Thus, we have that for every second $t$ elapsed the circle travels $1$ unit of arc segment with respect to its
circumference and $1$ unit of the segment with respect to the coordinate axis $x$. Then all $n\in \mathbb{N}$ will be associated with a respective arc segment and an angle correspon-\ ding to that arc segment. The Figure \eqref{Figura: $2$} presents, as a way of example, the position and angle traveled by the circle at $t=5$, positioning itself on the natural $n=5$.\\
\begin{figure}[ht]
\begin{center}
\includegraphics[width=10cm, height=7cm]{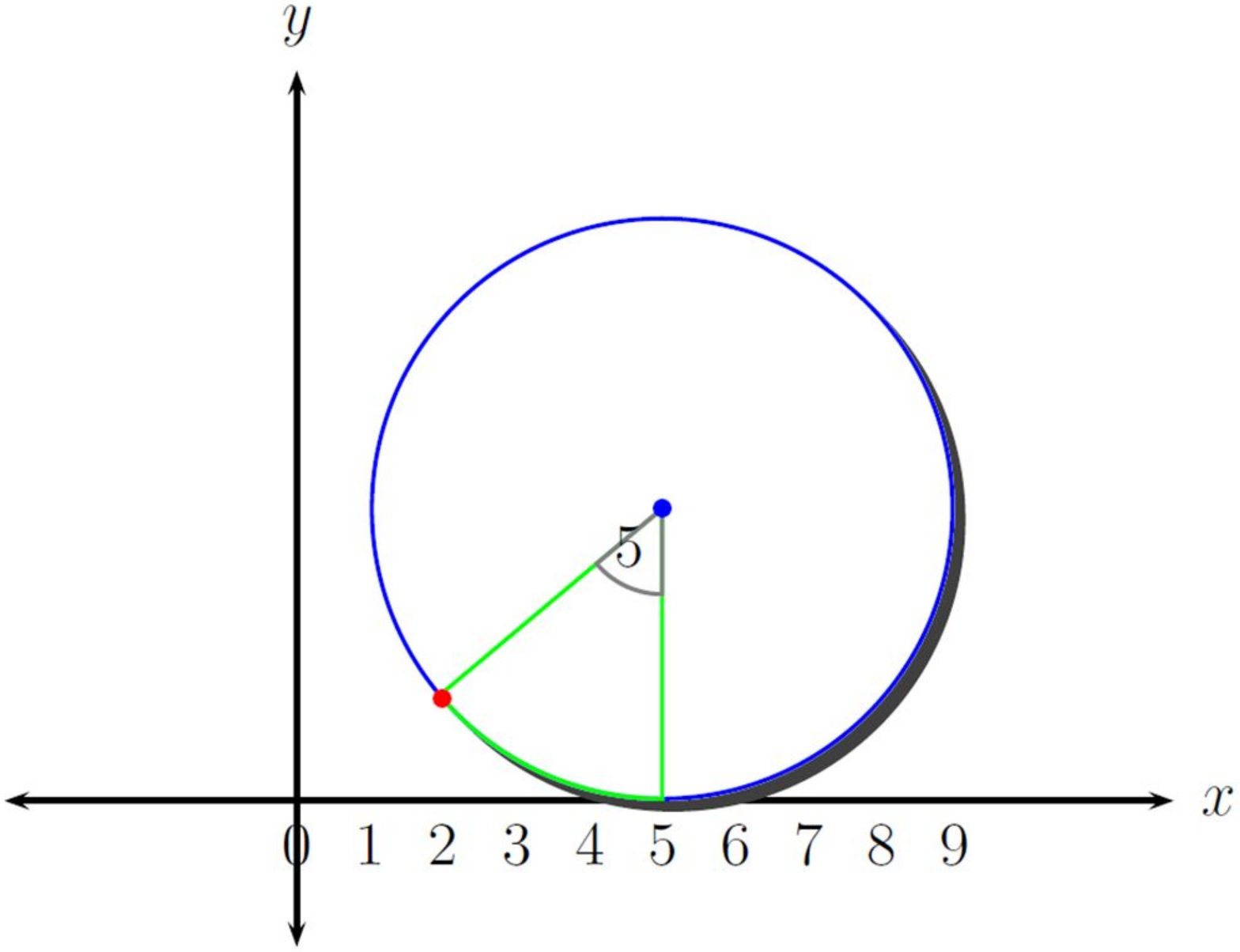}
\caption{Moving circle.}\label{Figura: $2$}
\end{center}
\end{figure}

\emph{Note}: it is clear that if $a,x \in \mathbb {Z}$ the equation $a \equiv x \pmod {360}$ has infinite congruent solutions for $x$. Now if $a$ is a prime number, then there is only one congruent solution $x$, being $x$ the remainder that leaves $a$ when it is divided by $360$.
$$x=min\,\{x \in \mathbb{N}: a\equiv x \pmod{360}\,\,\,\,\,;\,\,0\leq x<360\}.$$

\begin{defin}[Ova-angular residue de $\rho$]\label{def1}
Let $\mathbb{P}$ be the set of prime numbers and $\rho \in \mathbb{P}$. The solution for $x$ of the equation $\rho \equiv x \pmod{360}$
it will be called \emph{Ova-angular of $\rho$} and will be denoted by $\daleth_\rho=x$ such that:
$$\daleth_\rho = \rho-360 \left \lfloor \frac{\rho}{360} \right \rfloor,$$
equivalent
$$\rho \equiv \daleth_{\rho} \pmod{360}.$$
\end{defin}

Notice the equivalence between the angle and the residue. Figure \eqref{Figura: $3$} provides two examples when the circle is positioned on the prime numbers $359 $ and $367$ respectively.

\begin{figure}[ht]
\begin{center}
\includegraphics[width=14cm, height=9cm]{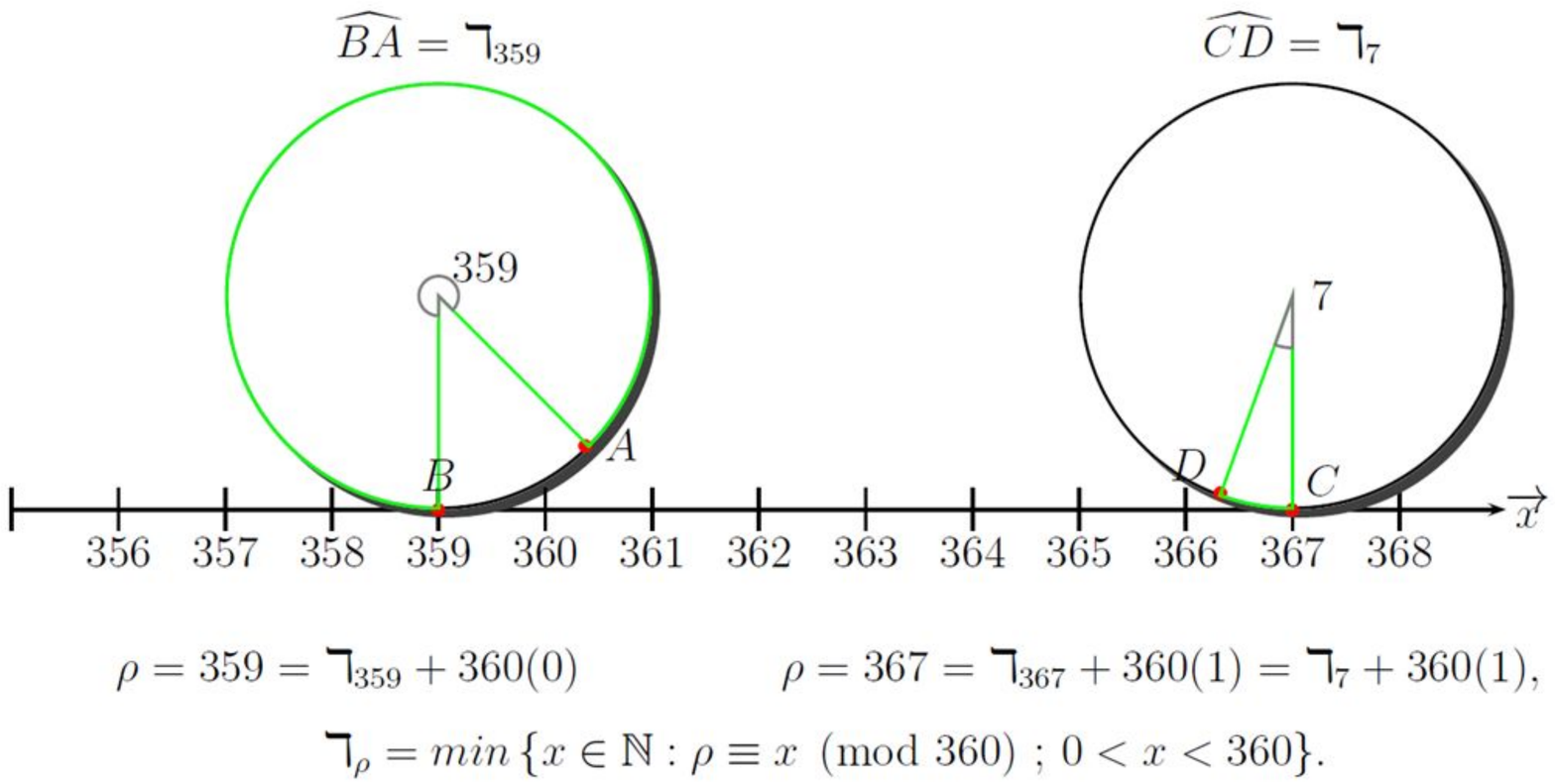}
\caption{Examples ova-angular rotation.}\label{Figura: $3$}
\end{center}
\end{figure}

\begin{defin}[Perfect Ova-angular set]
$\daleth_\rho$ is said to be an element of set A called
Perfect Ova-angular set, if $\daleth_\rho\in \mathbb{P}$.
$$A=\{\daleth_\rho\in \mathbb{N}\diagup \,\, \rho \equiv \daleth_{\rho}\pmod{360}\,\wedge \,\,
\daleth_{\rho}\,\,\text{is a  prime number};\, 0<\daleth_\rho<360\}.$$
\end{defin}

$A=\{2,3,5,7,11,13,17,19,23,29,31,37,41,43,47,53,59,61,67,71,73,79,83,89,97,101,\\
103,107,109,113,127,131,137,139,149,151,157,163,167,173,179,181,191,193,197,199,\\
211,223,227,229,233,239,241,251,257,263,269,271,277,281,283,293,307,311,313,317,\\
331,337,347,349,353,359 \}.$
$$|A|=card(A)=72.$$

\begin{defin}[Ova-angular set generator]
$\daleth_\rho$ is said to be an element of set B called Ova-angular set generator, if $\daleth_\rho \notin \mathbb{P}$.
$$B=\{\daleth_\rho\in \mathbb{N}\diagup \,\, \rho \equiv \daleth_{\rho}\pmod{360}\,\wedge \,\,
\daleth_{\rho}\,\,\text{is not a prime number};\, 0<\daleth_\rho<360\}.$$
\end{defin}

$B=\{1,49,77,91,119,121,133,143,161,169,187,203,209,217,221,247,253,259,287,289,\\
299,301,319,323,329,341,343\}.$
$$|B|=card(B)=27.$$

It is clear that the sets $A$ and $B$ are disjoint. It is striking that their cardinals are permutation
of their digits.

\begin{defin}[Residual set $\rho \pmod{360}$]
Let $C^{*}$ be the set formed by all the residuals left by a prime number when divided by $360$, i.e:
$$C^{*}=\{\daleth_\rho\in\mathbb{N} \diagup \,\, \daleth_\rho=\rho-360k\,;\,
\text{for some}\,\, k\in\mathbb{Z};\,\, 0<\daleth_\rho<360\}.$$
\end{defin}

It is clear that $C^{*}=A \bigcup {B}$ and that $C^{*}$ is a complete set of residuals of $\mathbb{P}_{360}.$

\begin{teor}[$Card(C^{*})=99$]
Every prime number $\rho$, when divided by $360$, leaves a positive integer residue $\daleth_\rho$ belonging to a set $C^{*}$ of only $99$ elements.
\end{teor}

\begin{proof}
There are several long and short ways to prove this using Dirichlet theorem or the properties of modular congruence, but is easier to observe that $|C^{*}|=|A|+|B|$.
\end{proof}

\begin{defin}[$\rho$ frecuency rotation]\label{Def:Frecuencia}
if $\rho \in \mathbb{P}$, then is said that its frequency of rotation denoted $\Game_\rho$, is given by the integer part of $\rho$ when divided by $360$.
$$\Game_\rho = \left \lfloor \frac{\rho}{360} \right \rfloor.$$
\end{defin}

From Definitons \eqref{def1} - \eqref{Def:Frecuencia} 
it is true that \begin{equation}\label{ecprincipal}
\forall\,\rho\,\in\mathbb{P}\,\,\,\,, \,\,\,\rho=\daleth_\rho+360(\Game_\rho).  
\end{equation}

Thus, we have that for every $\Game_\rho \in \mathbb{N}$ there are $\rho_{n, m, s...}$ and $\daleth_{\rho_{n, m, s, ...}}$ respectively,
such that the equation \eqref{ecprincipal} forms a prime number. Some examples: $1129=49+360(3)$, $3733=133+360(10)$,  $9161=161+360(25)$.\\

\begin{teor}[Ova-angular digits $1$]
In every prime number $\rho$, the last digit of the units corresponds to the last digit of the units of its Ova-angular $\daleth\rho$ respective.
\end{teor}

\begin{proof}
Let $\rho=\daleth_{\rho}+360(\Game_\rho)$ be a prime number. Note that regardless of the rotation frequency $\Game_\rho$, the  digit $0$ of $360$ guarantees that $360(\Game_\rho)$
it will always be $0$ with respect to the units, then in the sum with $\daleth_ {\rho}$ the last digit will always be the units respective of $\daleth_{\rho}$.
\end{proof}

e.g. $\rho=1619$, $\daleth_\rho=179$, $\Game_\rho=4$,
\,\,note that the units digit is common for both $\rho$ and $ \daleth_\rho$. \\

\begin{teor}[Ova-angular digits $2$]\label{Teor3} If the rotation frequency $\Game_\rho$ of $\rho$ is a multiple of $5$, then the last two digits of $\rho$ correspond to those of $\daleth_\rho$ respectively.
\end{teor}
\begin{proof} Let $\rho=\daleth_{\rho}+360(\Game_\rho)$ be a prime number with $\Game_{\rho}=5t$ and $t\in\mathbb{N}$. It is clear that $1800t$ always ends in $00$. It is concluded that the digits corresponding to the respective tens and ones of $\ rho$ and $\daleth_\rho$ are equal.
\end{proof}

e.g. $\rho=7537$, $\daleth_\rho=337$, $\Game_\rho=20$.

\begin{corol}[Ova-angulares digits 3] If the rotation frequency $\Game_\rho$ of $\rho$ is a multiple of $25$ then the last three digits of $\rho$ will be $\daleth_\rho$ directly.
\end{corol}

\begin{proof}
The proof is analogous to the previous Theorem.
\end{proof}

e.g: $\rho=36299$, $\daleth_\rho=299$, $\Game_\rho=100$.

\section{Geometry and applications of the Ova-angular rotation circle.}

\textbf{Note $\textbf{1}$.} In basic number theory, there is a theorem which establishes that for all $n \in \mathbb{N}$ there are $n$ consecutive integers and compounds, it should be clarified that if the Ova-angular rotation circle is positioned in any integer belonging to the interval $[(n + 1)! + 2, (n + 1)! + n + 1]$ it has not a prime number, therefore it will not be an object of interest, it will only rotate. This suggests that:
\begin{teor}\label{teor intervalos1}
Let $I_n\subset \mathbb{N}$ be intervals of $n$ elements where there is definitely no prime numbers.
$$I_n=[ (n+1)!+2,(n+1)!+n+1] \,\,\text{for all}\,\, n\in \mathbb{N}.$$
\end{teor}

\begin{proof}
For all $n\in \mathbb{N}$ is true that $(n+1)!+2,(n+1)!+3,..., (n+1)!+n+1$ are compound numbers.
\end{proof}
At these intervals, the Ova-angular rotation circle will make transitions or jumps.

\begin{corol}\label{corol intervalos2}
For all $n \in  \mathbb{N}$ there is always at least a prime number $\rho$ such that $(n+1)!+n+1 <\rho < (n+2)!+2$.
\end{corol}

\begin{corol}\label{corol intervalos3}
The distance $D$ between the intervals is $D=I_{n+1} - I_{n}= (n+1)^2 n!-n+1 $ for all $n \in \mathbb{N}$.
\end{corol}

These results state that for all $n$, there are intervals of $n$ consecutive integers composed by which the circle rotates without interest, and furthermore, there are $(n +1)^{2}n!-n+1$ consecutive integers that are possible primes where the circle rotates. Thus, since $n<(n+ 1)^{2}n!-n+1$ it is possible to affirm that the distance between a prime and it's next it is not that it increases infinitely, but rather that it restarts, taking extensive breaks in every $I_n$ interval (see figure \eqref {Figura: It}).\\

\begin{figure}[ht]
\begin{center}
\includegraphics[width=10cm, height=5cm]{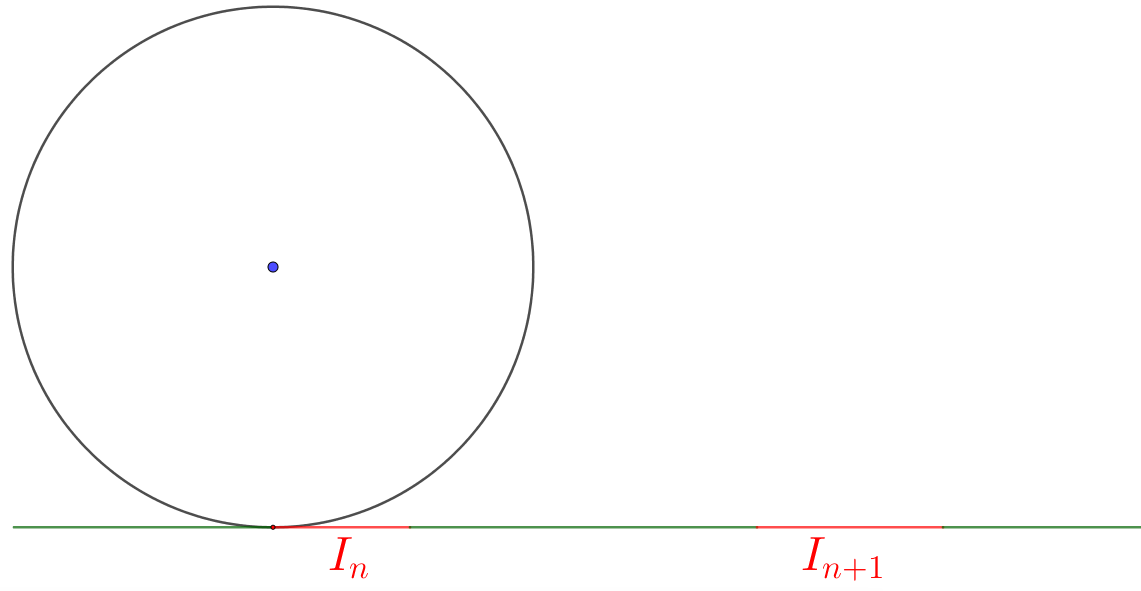}
\caption{
Ova-angular rotation circle transition.}\label{Figura: It}
\end{center}
\end{figure}

On the other hand, just as the intervals where there are no prime numbers grow infinitely, the intervals where there are prime numbers also grow infinitely, in fact, the interval in which there are prime numbers will always be greater than where there is not. This allows us to establish that:

\begin{corol}\label{corol intervalos4} For all $n \in \mathbb{N}$, there are $n$ consecutive odd integers in which at least one of them is a prime number.
\end{corol}

\textbf{Note $\textbf{2}$.}
What is being asserted is not the fact that for all rotation $\Game$ that is between two intervals $I_n$ a prime number exists, no; rather than that, in the said space there is at least one consecutive $\Game$ and $\Game+1$ rotation in which primes are formed.
\begin{teor}
For all $\daleth_\rho \neq 2,3,5$ is true that:
\begin{equation}
\daleth_{\rho}^{\varphi(360)} \cong 1 \pmod{360}.
\end{equation}
\end{teor}

\begin{proof}
Just apply Euler's theorem for all $\daleth_\rho \neq 2,3,5 $, since $360=2^{3}.\,3^{2}.\,5^{1}$.
\end{proof}

The above suggests that the multiplicative inverse that the system solves
$\daleth_{\rho} \,X \cong 1 \pmod{360}  \,,  \,\,\,\textit{with}\,\,\daleth_\rho \neq 2,3,5$, is $X={\daleth_{\rho}}^{\varphi(360)-1}=\daleth_{\rho}^{95}$, with $\varphi$ Euler's function.

\begin{teor}[Equivalence relation]
The prime numbers under the Ova-angular rotations fulfill an equivalence relation.
\end{teor}
\begin{proof}
\emph{Reflexive:} is clear that $\rho\equiv\rho \pmod{360}\,\, \Leftrightarrow \,\,\daleth_\rho \equiv \daleth_\rho \pmod{360}\,\, \,\,\,;\,\,\forall\,\rho\,\in \mathbb{P}.$\\
it is clear that the properties of symmetry and transitivity are also fulfilled.\\
\end{proof}

\begin{corol}[$\mathbb{P}$ partition] Ova-angular rotations generate a partition of $\,\mathbb {P}$ in equivalence classes. i.e $\mathbb{P}_{360}$. 
\end{corol}

In order to analyze the geometry in the rotation circle, all of the above is summarized with the following definition:

\begin{defin}[Ova-angular function]
Let $f_{\pmod{360}}:$ $\mathbb{P} \rightarrow C^{*}$ be
such that if $\rho \in \mathbb{P}$ then
$f(\rho)=\daleth_{\rho}\,\,\,\,\text{,with}\,\,\rho=\daleth_{\rho}+360(\Game_{\rho})$ and $0<\daleth_{\rho}<360.$
\end{defin}
It is clear that the function $f$ is well defined, in particular $f$ is surjective.\\

From Dirichlet's theorem "on arithmetic progressions" is clear that all Ova-angular residues $\daleth_\rho\neq 2,3,5$ 
produce infinite prime numbers, highlighting with greater interest $2,3$ and $5$, unique primes in their respective class. Finally, it is highlighted that all Ova-angular $\daleth_\rho \neq 2,3,5$ is itself a unit of $\mathbb {P}_{360}$ and and forms a reduced residue system (R.R.S).\\

Below are some geometric properties of the Ova-angular rotation circle.

\subsection{\textbf{Geometry}} \definecolor{Micolor4}{gray}{1.0}
\textcolor{Micolor4}{.}\newline
In Figure \eqref{Figura: circle1} all the elements of the set $C^{*}$ have been located. It can be seen that the behavior of the prime numbers is not disordered as is usually conjectured if it is carefully observed by locating a vertical axis of symmetry that passes through the center of the circle. As shown in figure \eqref{Figura: circle2} we can visualize that there is a symmetry between the $\daleth_\rho<180$ (Right) and the largest at $\daleth_\rho>180$ (Left).\\
\begin{figure}[ht]
\raggedright
\includegraphics[width=1.11\textwidth]{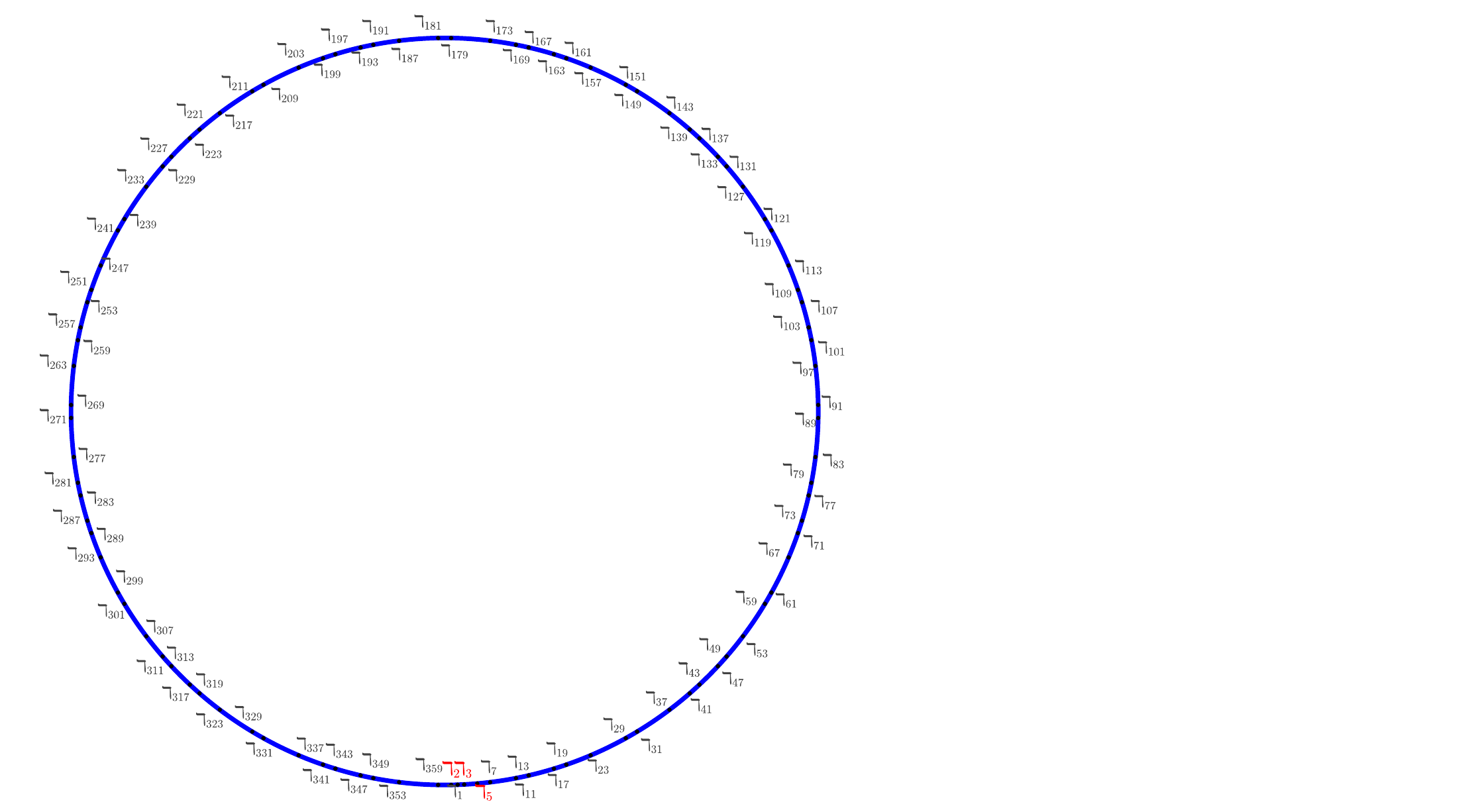}
\caption{Ova-angular residues.}\label{Figura: circle1}
\end{figure}

\begin{figure}[ht]
\centering
\includegraphics[width=1.11\textwidth]{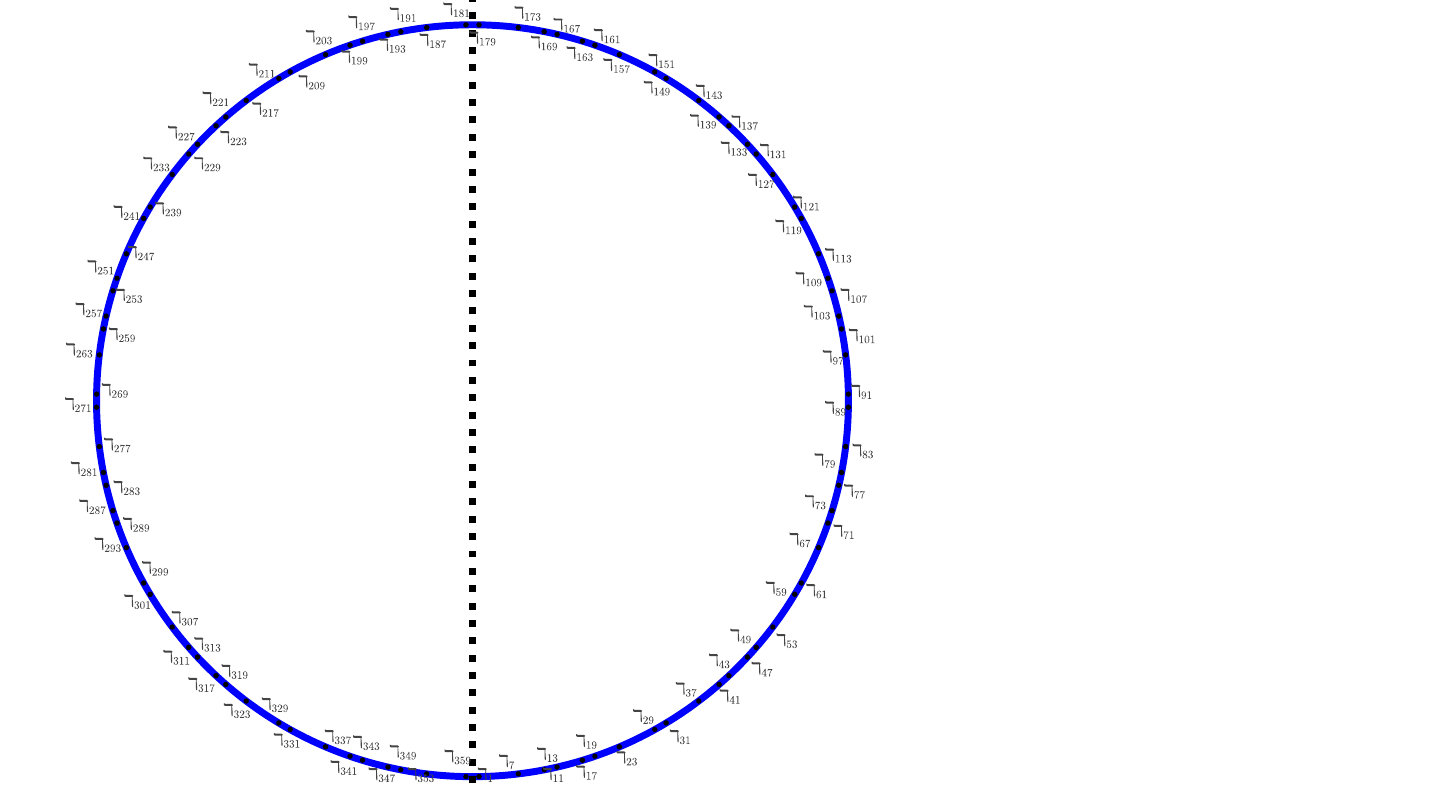}
\caption{Ova-angular residues division.}\label{Figura: circle2}
\end{figure}

Hereinafter, the set $C=C^{*}-\{2,3,5\}$ will be considered.\\

In Figure \eqref{Figura:circle4} only the consecutive Ova-angular residues $\daleth_\rho$ are taken with a difference of two units (\emph {Twin lines}). These $\daleth_\rho$
correspond to residual classes two to two that generate twin prime numbers. \\

\begin{figure}[ht]
\centering
\includegraphics[width=1.1492\textwidth]{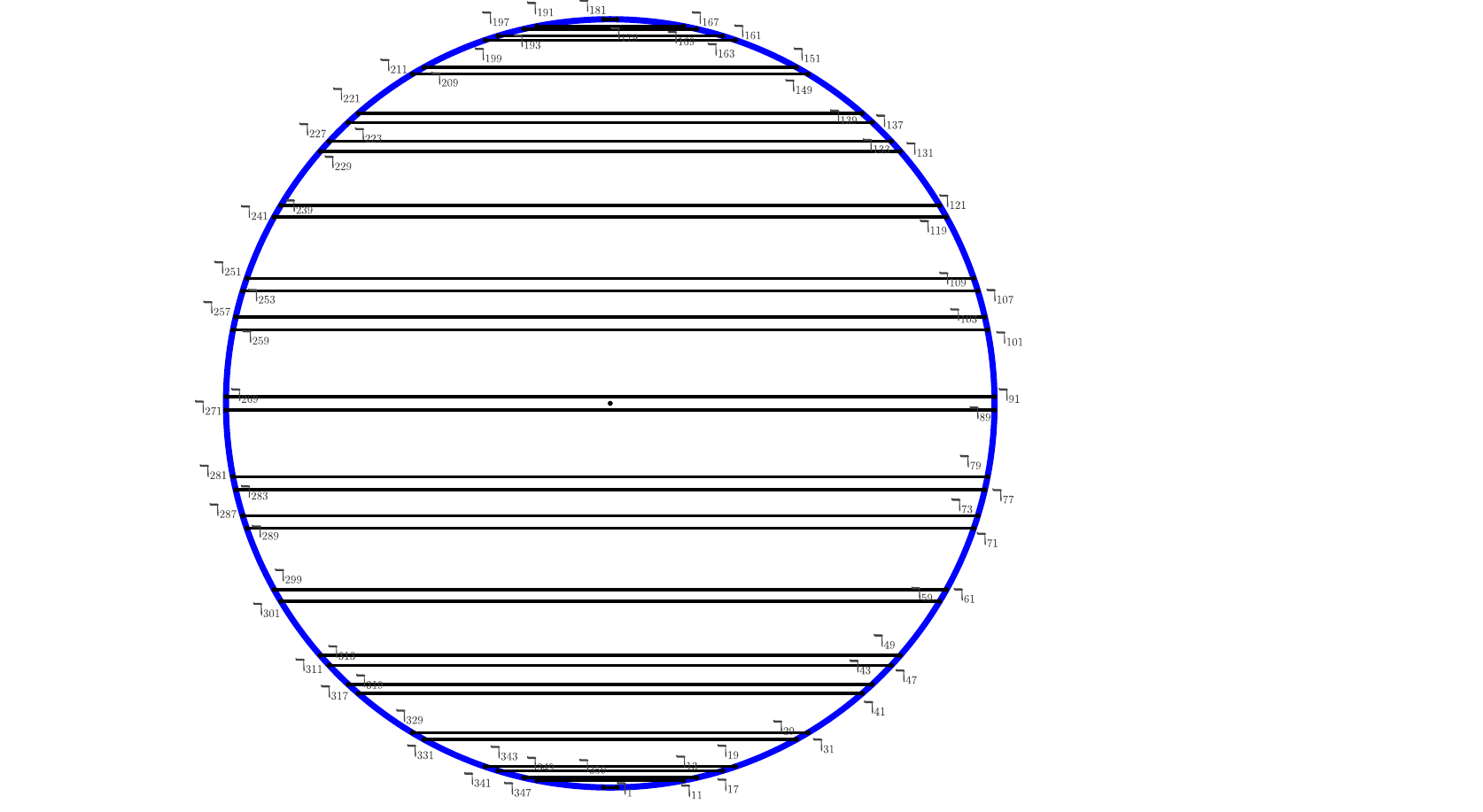}
\caption{Twin lines.}\label{Figura:circle4}
\end{figure}

Similarly, in Figure \eqref{Figura:circle5} the nonconsecutive Ova-angular residues $\daleth_\rho$ are taken (\emph{Particular lines}).\\

\begin{figure}[ht]
\centering
\includegraphics[width=1.1492\textwidth]{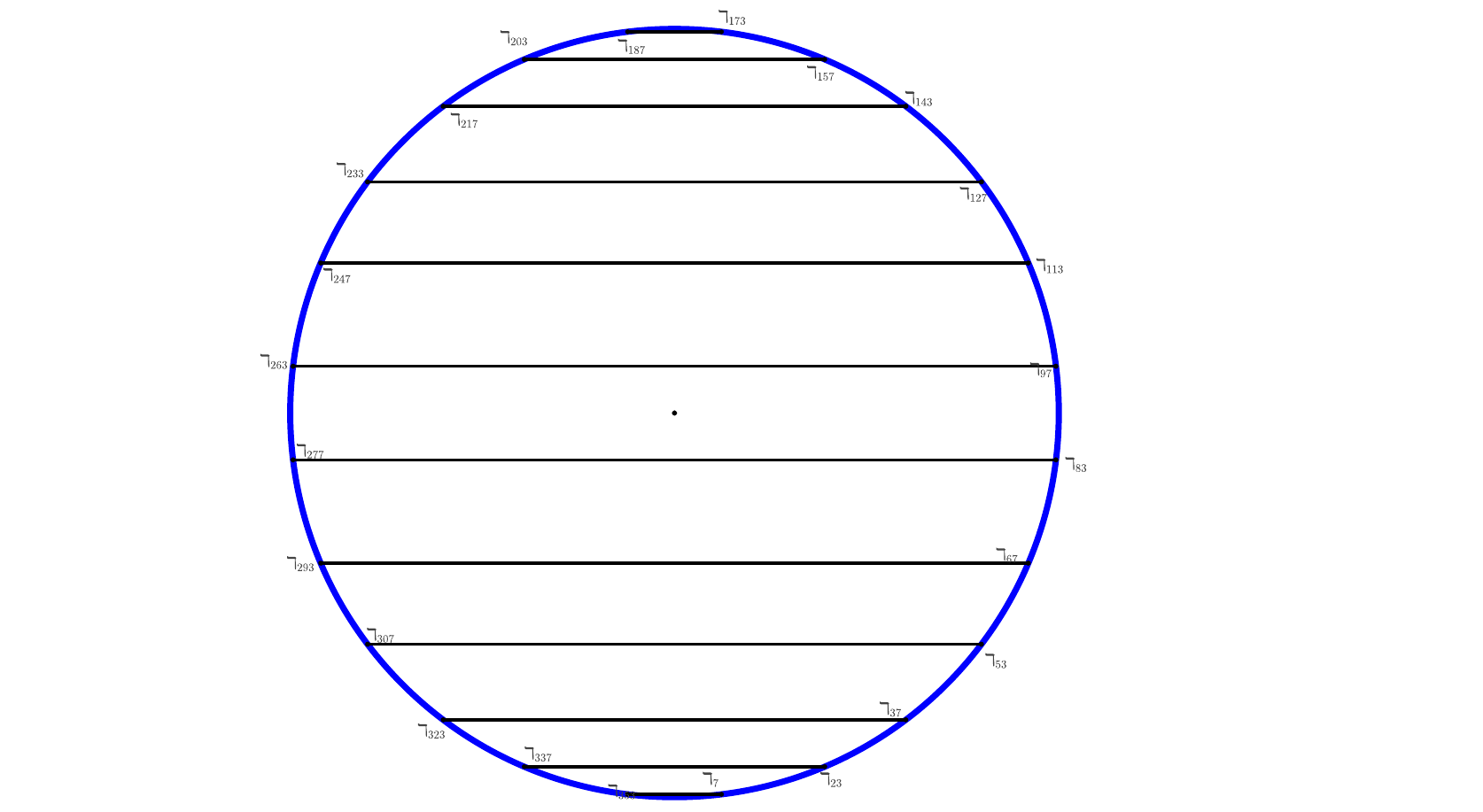}
\caption{Particular lines.}\label{Figura:circle5}
\end{figure}

The Figures \eqref{Figura:circle6}, \eqref{Figura:circle7} y \eqref{Figura:circle8} direct the lines described above to the center to provide an aesthetic look of the Ova-angular residues. 

\begin{figure}[ht]
\centering
\includegraphics[width=0.7142\textwidth]{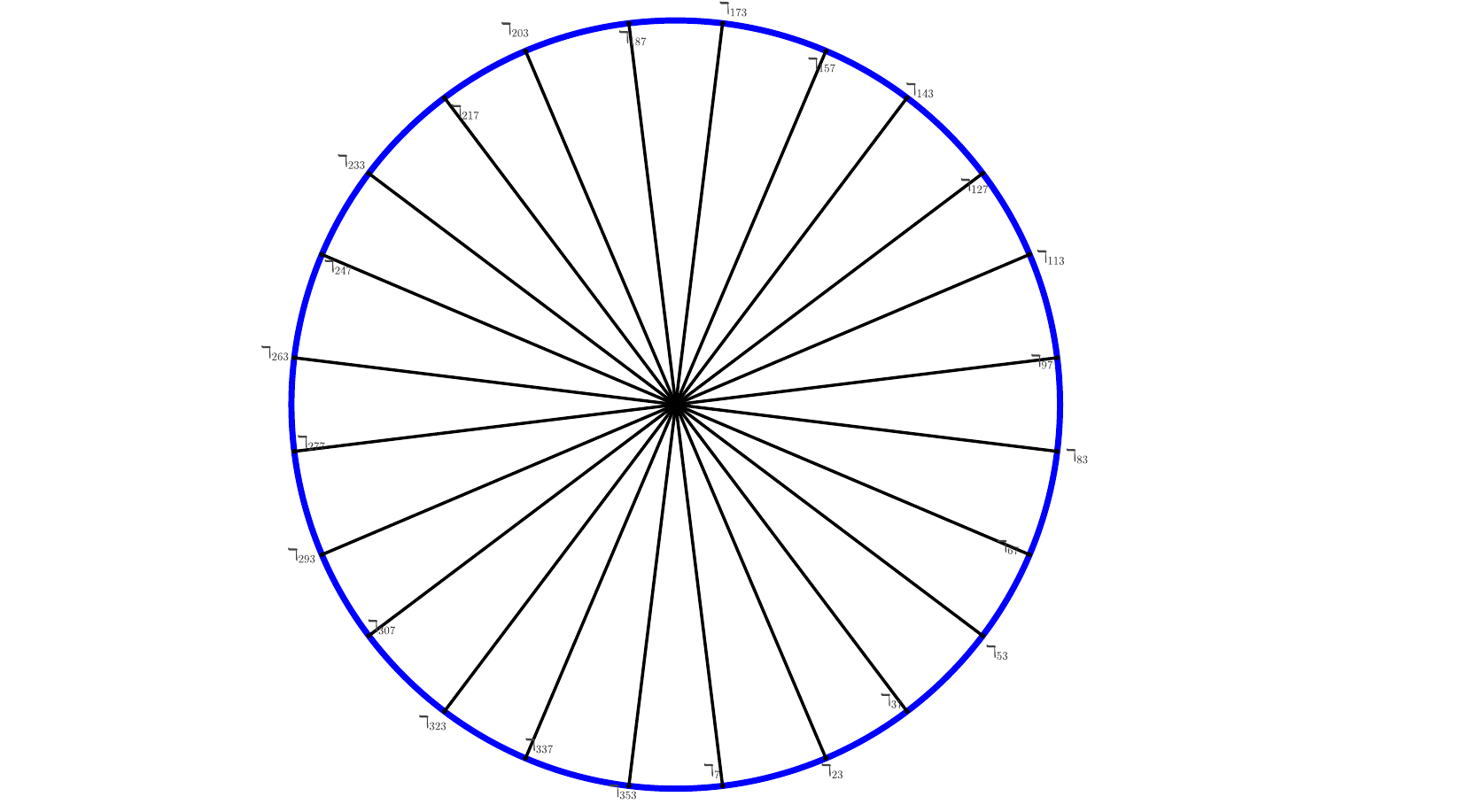}
\caption{Particular lines to the center.}\label{Figura:circle6}
\end{figure}

\begin{figure}[ht]
\centering
\includegraphics[width=0.7142\textwidth]{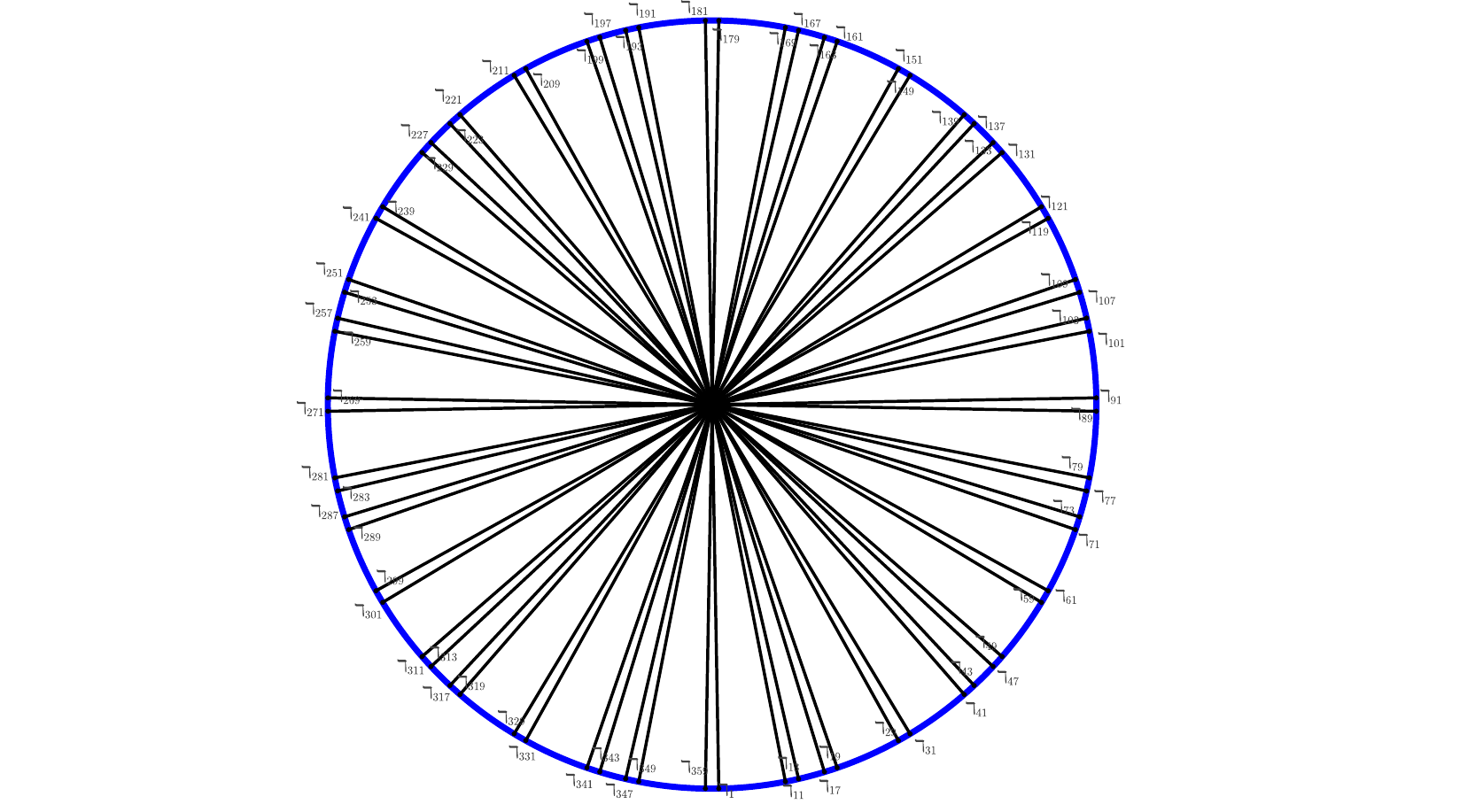}
\caption{Twin lines to the center.}\label{Figura:circle7}
\end{figure}

\begin{figure}[ht]
\centering
\includegraphics[width=0.7142\textwidth]{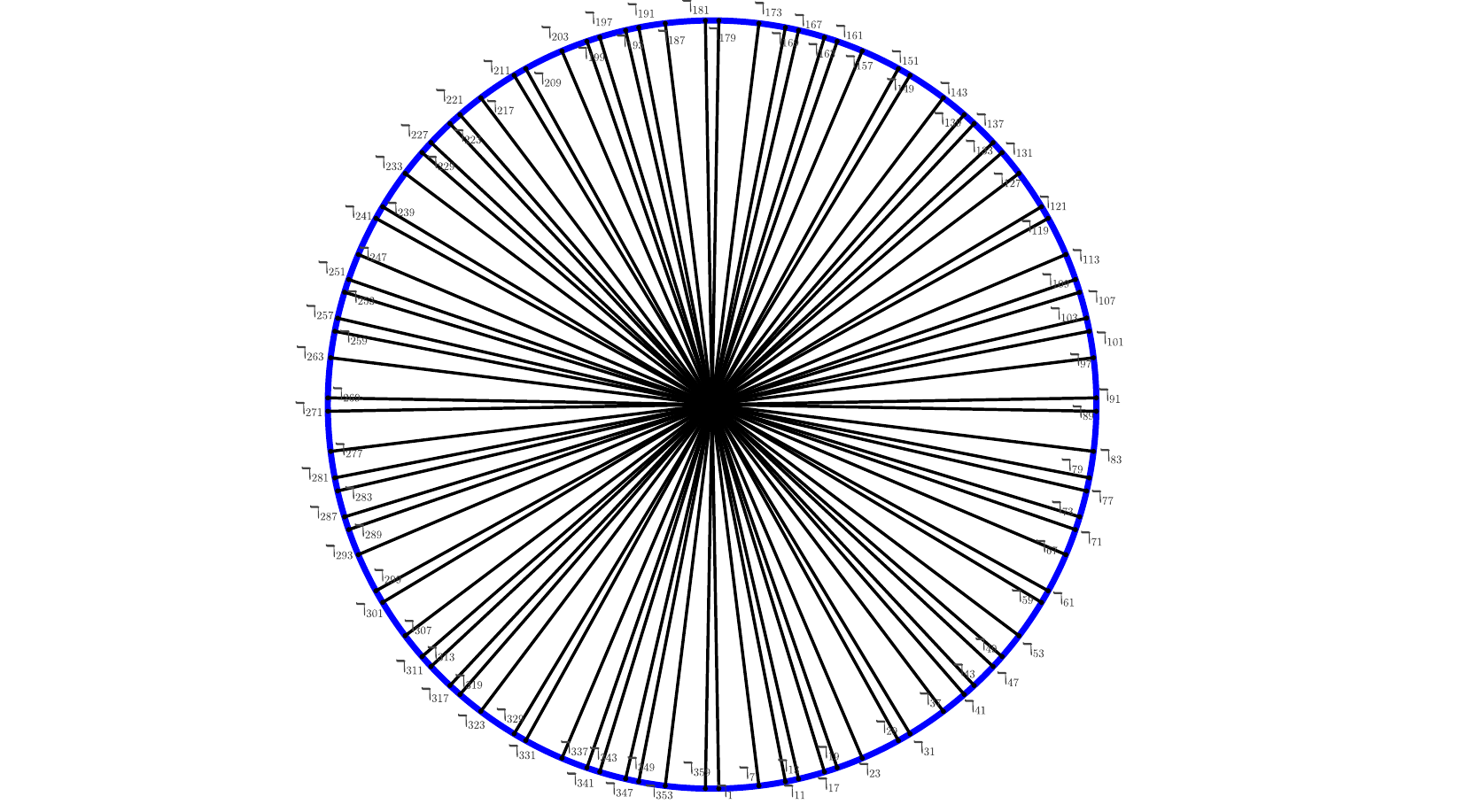}
\caption{\,\,\,\,\,\,All lines to the center.}\label{Figura:circle8}

\end{figure}

\clearpage


\begin{teor}[$C$ elements]
For all $\daleth_{\rho} \in C$ is true that: i) $360-\daleth_{\rho} \in C$, ii) $\daleth_{\rho_i}*\daleth_{\rho_k} \pmod{360} \in C$, iii) $\daleth_\rho^n \pmod{360} \in C$ for all $n \in \mathbb{N}$. \end{teor}

\begin{proof}
i) and ii) The proof is trivial. Just analyze for each element of $ C $.\\

iii)(Matematical induction.) Let be $\daleth \in C$, it is clear that for $n=1$ it is true that
$$\daleth^1 \equiv \daleth \pmod{360} \,\,;\,\,   \textit{for some} \,\,\daleth \in C, $$
suppose it is true for $n=k$, that is:
$$\daleth^k \equiv \daleth_s \pmod{360} \,\,;\,\, \textit{for some} \,\, \daleth_s \in C \,\,\,\,(\textit{
Inductive hypothesis}).$$
However
$$\daleth^{k+1} \equiv \daleth^k*\daleth \pmod{360}, $$
then by the inductive hypothesis we have
$$\daleth^{k+1} \equiv \daleth_s*\daleth \pmod{360} \,\,;\,\, \textit{for some} \,\, \daleth_s \in C, $$
Finally by ii) it is true that
$$\daleth^{k+1} \equiv \daleth_r \pmod{360} \,\,;\,\, \textit{for some} \,\, \daleth_r \in C. $$
\end{proof}

The above explains the symmetrical sense of the previous Figures.

\begin{teor}\label{elementoinverso}
For all $\daleth_{\rho} \in C$ is true that the inverse of the $\daleth$ is in $C$; i.e $\daleth^{-1} \in C$.
\end{teor}

\begin{proof}
The proof is trivial if all the inverses for each $\daleth \in C $ are found.
\end{proof}

\textbf{Note $\textbf{3}$.} It is highlighted that the Ova-angular rotation circle establishes a topologically stable set, it presents a completely circular symmetry, that is symmetry about the $x$ axis, $y$ axis, the lines $y =x $, and $y =-x$ u oblique. The Figures \eqref{Figura:circle9}, \eqref{Figura:circle10} and \eqref{Figura:circle11} show these symmetries and the Figure \eqref{Figura:circle12} groups all these into a better image of the prime numbers generator set. Finally, the Figure \eqref{Figura:circleInversos} shows once again the symmetry in the Ova-angular circle by connecting each Ova $\daleth$ with its respective inverse $\daleth^{-1} \in C$.

\begin{figure}[ht]
\centering
\includegraphics[width=0.7\textwidth]{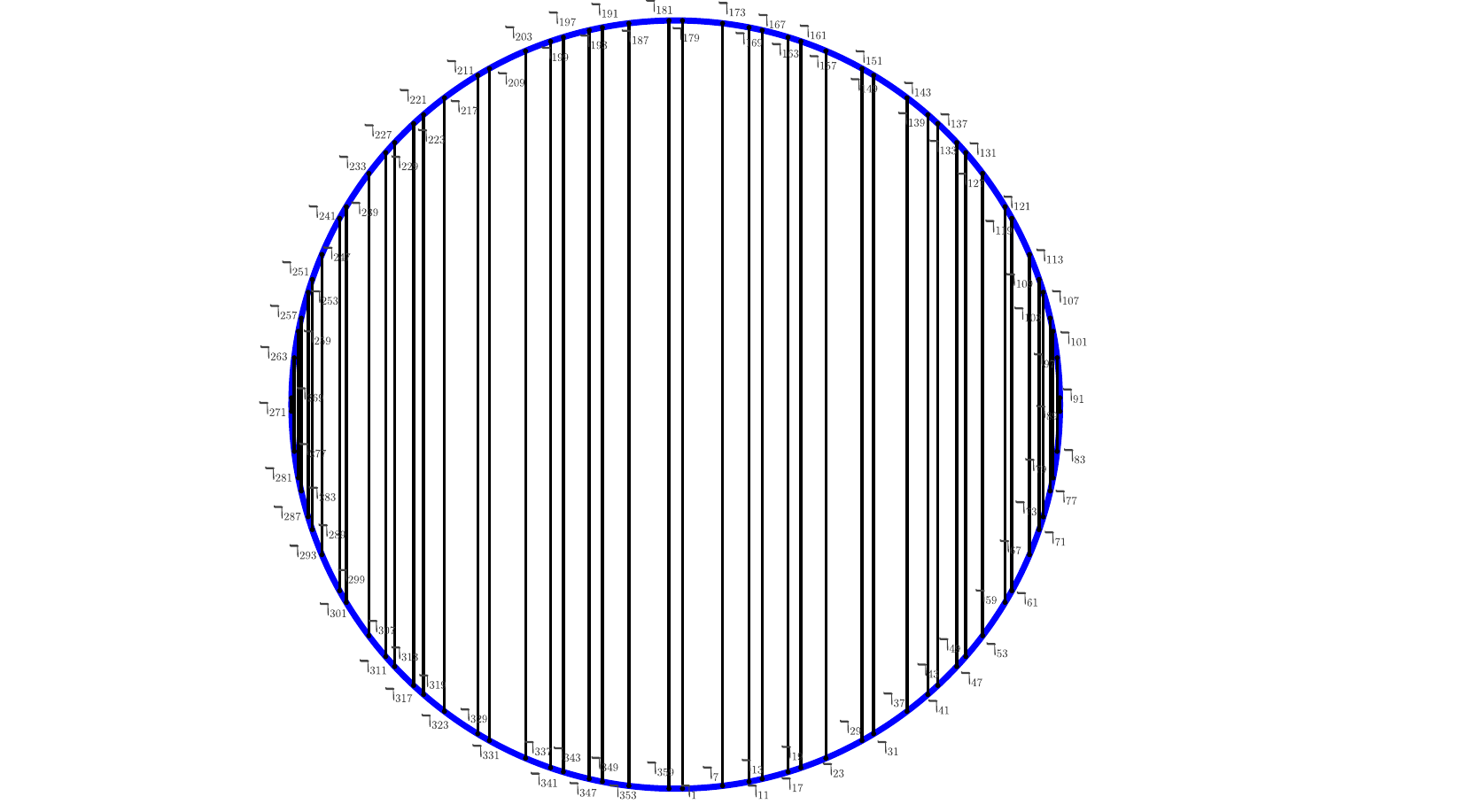}
\caption{x axis symmetry.}\label{Figura:circle9}
\end{figure}

\begin{figure}[ht]
\centering
\includegraphics[width=0.6\textwidth]{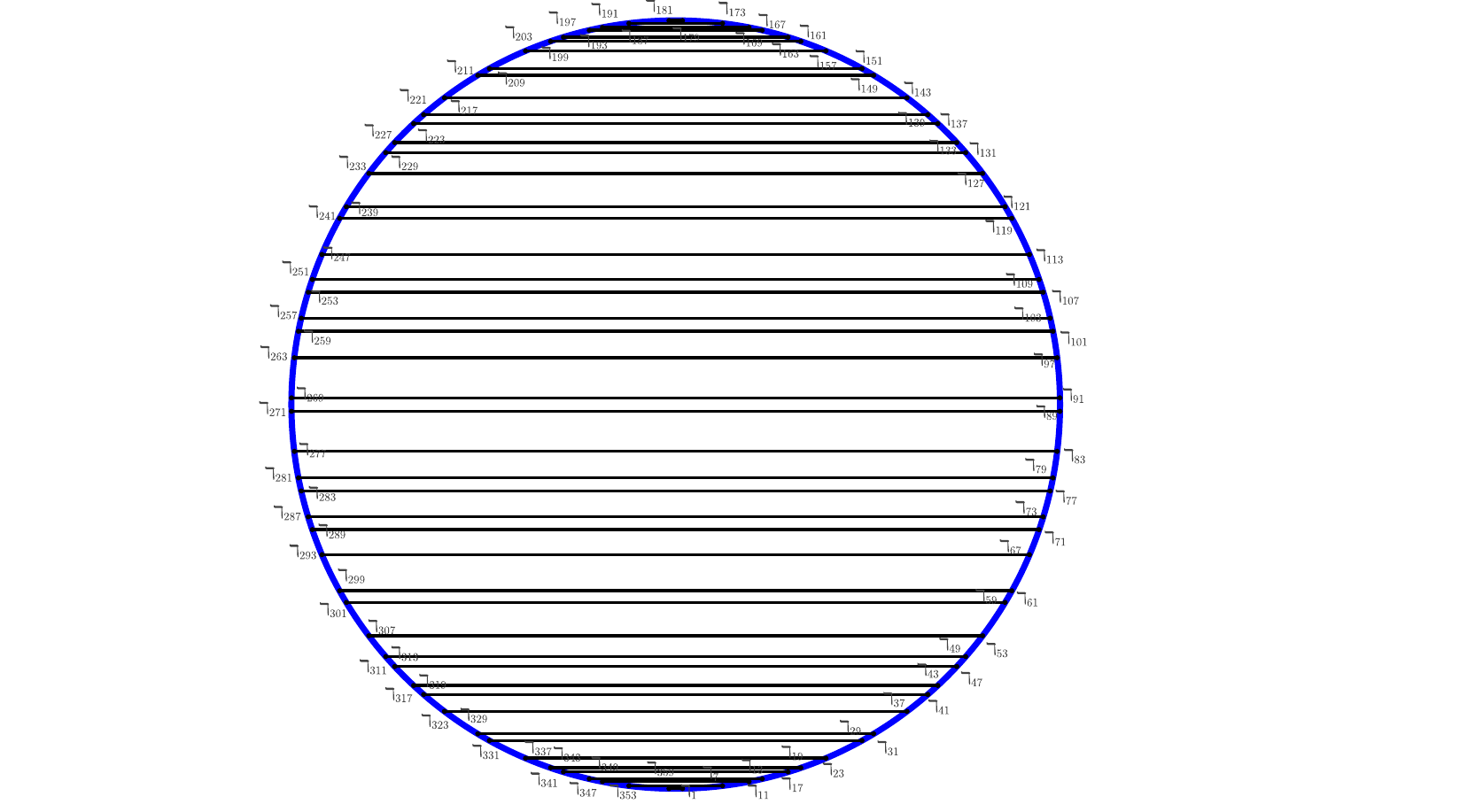}
\caption{y axis symmetry.}\label{Figura:circle10}
\end{figure}

\begin{figure}[ht]
\centering
\includegraphics[width=0.6\textwidth]{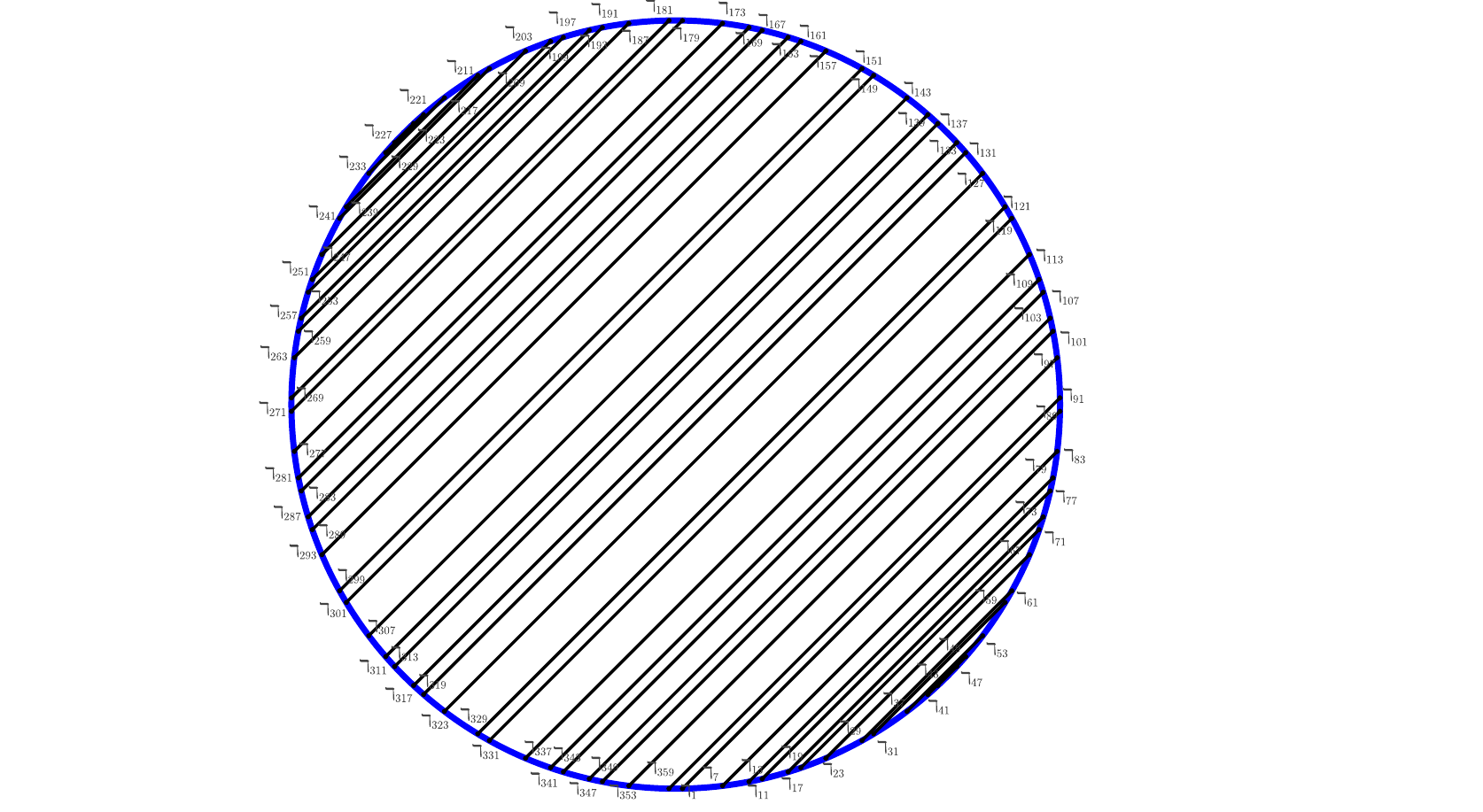}
\caption{Oblique symmetry.}\label{Figura:circle11}
\end{figure}

\begin{figure}[ht]
\centering
\includegraphics[width=0.803\textwidth]{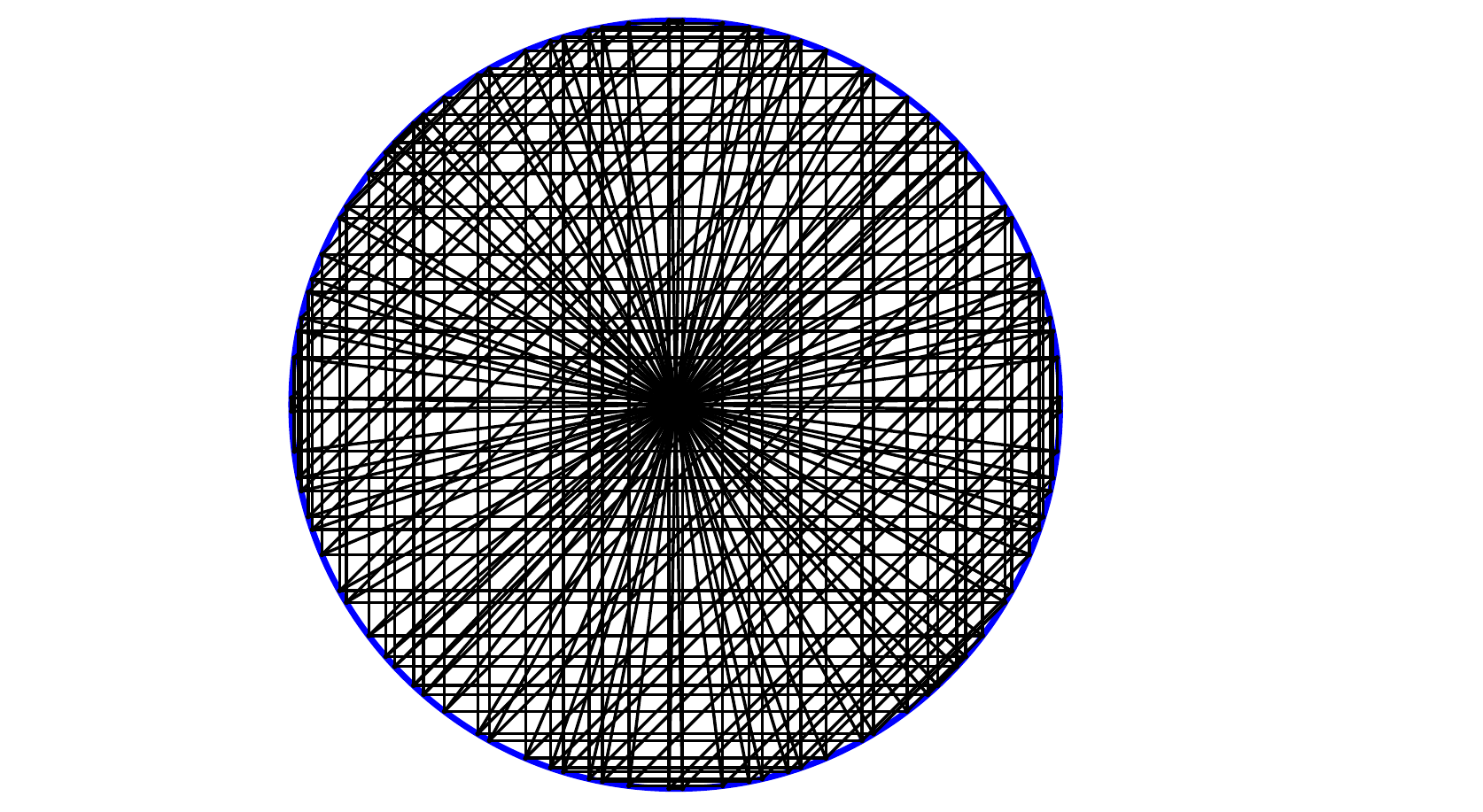}
\caption{Symmetries in the Ova-angular rotation circle.}\label{Figura:circle12}
\end{figure}
\begin{figure}[ht]
\centering
\includegraphics[width=1.5\textwidth]{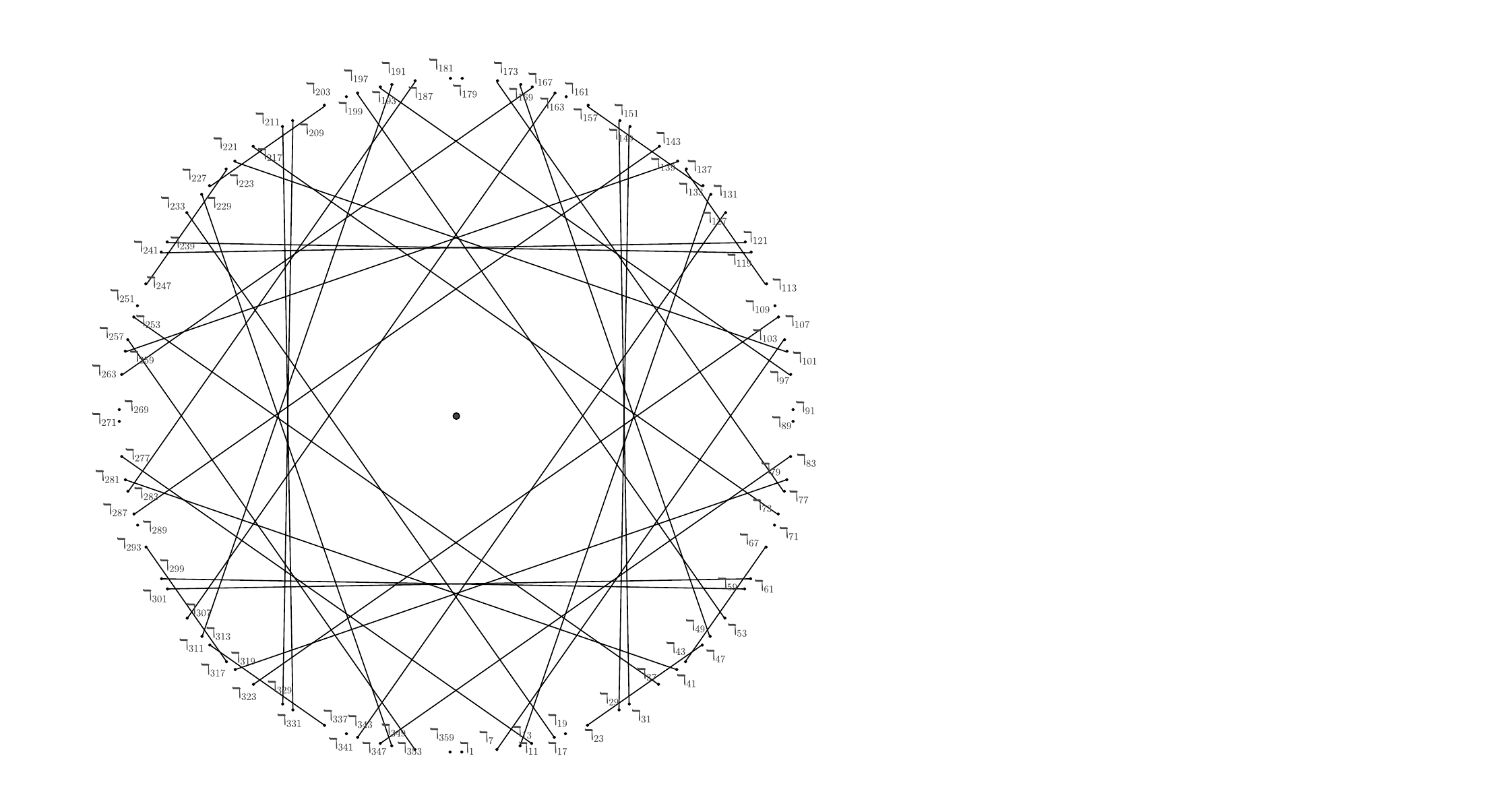}
\caption{Ova-inverse connection.}\label{Figura:circleInversos}
\end{figure}
\clearpage

All the previous Figures are regrouped in the following theorem.

\begin{teor} [Generating function] \definecolor{Micolor4}{gray}{1.0}
\textcolor{Micolor4}{.}\newline
\textbf{A)Particular lines:}
All $ \daleth \in C$ and positioned on the particular lines is a coefficient on the Taylor expansion of the generating function:
$$G_n(a_n)(z):=\frac{7z^2+16z+7}{z^3-z^2-z+1} \,\,;\,\,\,\,\,\,a_n=\frac{1}{2}\left ( 30n+(-1)^n-15) \right )  \mod(360).$$
\textbf{B)Twin lines:}
All $ \daleth \in C$ and positioned on the twin lines is a coefficient on the Taylor expansion of the generating function:
$$G_n(b_n)(z):=\frac{11+2z+4z^2+2z^3+10z^4+2z^5-z^6}{(z-1)^2(1+z+z^2+z^3+z^4+z^5)} \,\,;\,\,\,\,\,\,b_n \mod(360).$$

\textbf{C)Generating function:}
Every $\daleth \in C$ positioned on the rotation circle is a cyclical coefficient on the Taylor expansion of the generating function

$$G_n(c_n)(z):=\frac{7+4z+2z^2+4z^3+2z^4+4z^5+6z^6+2z^7-z^8}{(z-1)^2(1+z+z^2+z^3+z^4+z^5+z^6+z^7)} \,\,;\,\,\,\,\,\,c_n \mod(360).$$
\end{teor}
\begin{proof}
By independently estimating the succession of the particular lines and the twin lines, it is easy to obtain more information about them by obtaining their respective genera\-ting function using different methods \cite{NARUSE2018197}.
\end{proof}

\subsection{\textbf{Aplications}}\begin{teor}
For all frequency $\Game \in \mathbb{N}$ there is a prime $ \rho$ and therefore an ova $ \daleth_{\rho} \in C$ such that $\rho$ has a frequency $\Game \leq \Game_ 
{\rho} \leq 2 \Game + 1 $.
\end{teor}
\begin{proof}
Bertrand's theorem states that for every positive integer $n>1$ there is a prime $\rho$ such that $n<\rho<2n$, then the proof is immediate when calculating the frequencies of rotations.
\end{proof}

\begin{teor} [Ova-angular of primes of the form $n^2+1$.]
Let $C_{n^2+1}$ be the set formed by the Ova-angular $\daleth_\rho$ of the primes of the form $n^2 + 1$.
Then:
$$C_{n^2+1}=\{ 1, 2, 5, 17, 37, 41, 77, 101, 137, 161, 181, 197, 217, 221, 257, 281, 317, 341 \}.$$
\end{teor}

\begin{proof}
Let $\rho$ be a prime number such that $\rho=k^2+1$ with $k$ integer. It is clear that $\daleth_\rho-1 \equiv k^2 \pmod{360}$ then $\daleth_{\rho}-1 \equiv r^2 \pmod{360}$ for some $r<360$. Manually checking, for all $r<360$ results in all $\daleth_\rho$ that satisfies the modular congruence.
\end{proof}

The Figure \eqref{Figura:Landau} shows the connection of the residuals $\daleth_{n^2 +1}\neq 2,5$ in the rotation circle. It is worth highlighting the reflective symmetry that the Figure presents.
\clearpage

\begin{figure}[hbt!]
\centering
\includegraphics[width=1.7\textwidth]{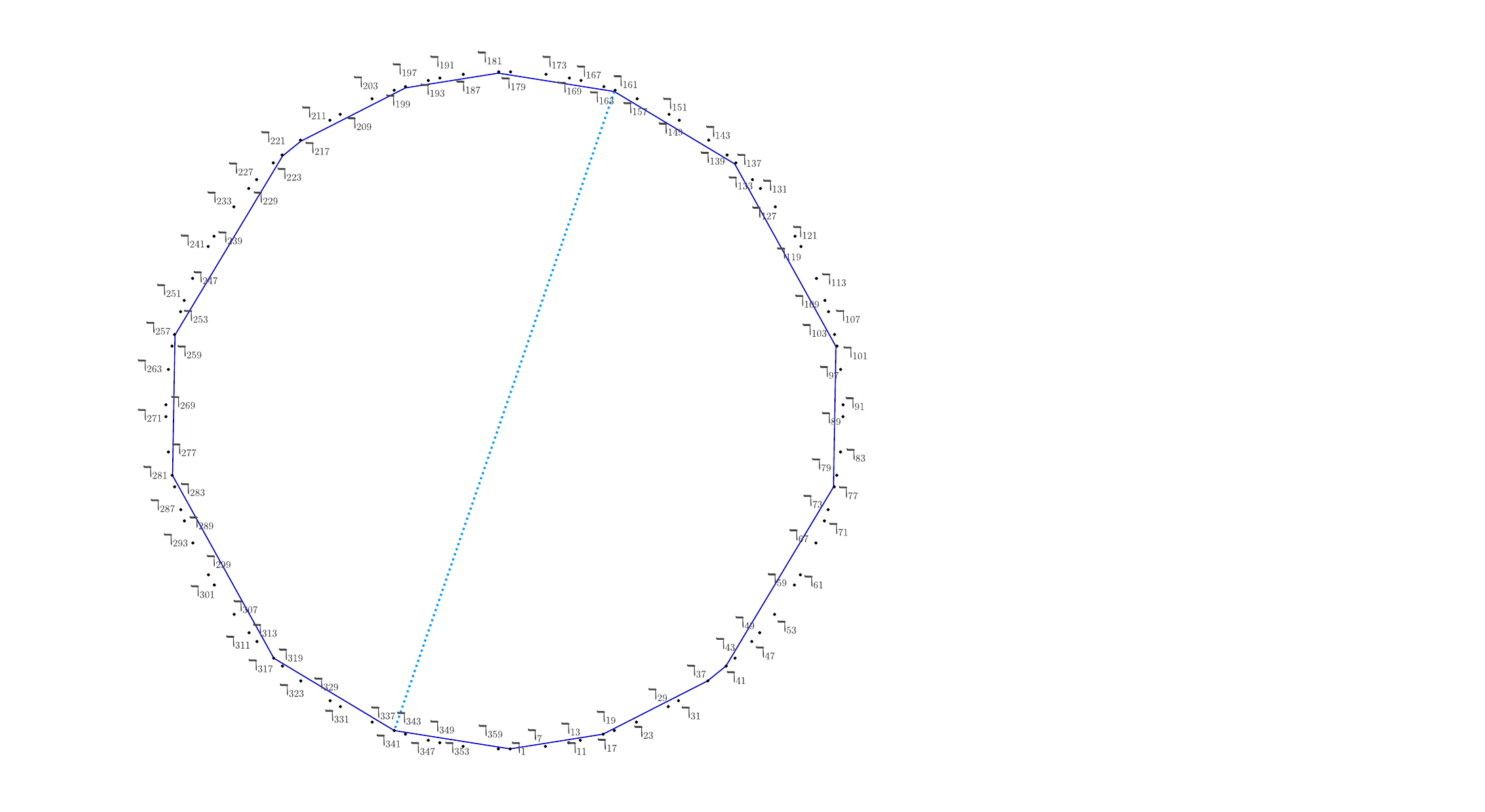}
\caption{\,\,\,\,\,\,Ova - $\daleth_{n^2+1}$ conection.}\label{Figura:Landau}
\end{figure}

\begin{teor}
There are infinite primes $\rho$, such that $\rho = k^2+1$ for some integer $k$.
\end{teor}

In accordance with the original Landau statement\footnote{Weisstein, Eric W. Landau is Problems. From MathWorld--A Wolfram Web Resource http://mathworld.wolfram.com/LandausProblems.html}, it is interesting to demonstrate if the list $2, 5, 17, 37, 101, 197, 257, 401, 577, 677, 1297, 1601,...$ OEIS A$002496$ is infinite or finite.

\begin{proof}
For this proof, we are going to take an element from the list and from this and its Ova-angular residue it will be shown that they are indeed infinite.\\

Let's arbitrarily take the prime number $1601=\daleth_{161} +360(4)$ this is $40^2+1=161+360(4)$. Now thinking about $n$ rotations, let's reason for the absurd (\emph{reduction to the absurd}), suppose that there is a last prime $\rho_n =161+360 (4+n)$ that is of the form $k^{2}+1$, that is, $\rho_n$ is the last Ova-angular prime of $161$ which is $1$ greater than a perfect square.\\

If it is possible to build a prime number that also meets the desired condition under the same ova-angular, we can reach some contradiction.\\

As $40^2+1+360n=161+360(4+n)$, from this expression we have that $40^2+360n=\rho_{n}-1$ or better yet, there is an integer $ k $ such that:
$$40^2+360n=k^2$$
$$n=\frac{(k+40)(k-40)}{360}$$
Then, under this procedure, if there is a function $k(\alpha)$ y $n(\alpha)$ of link between $k$ and $n$, with parameter $\alpha$ integer, then there will be infinite solutions to $n$ in the rotation system $\rho =161+360(4+n)$ and therefore, since there are infinite primes in the arithmetic progression $161+30(\Game)$ the proof would be complete. To understand the above consider the following: \\

Let be ${n}'>n$ and let be $z={n}'-n$ ; $z \in \mathbb{Z}$, then $161+360({n}')=161+360(z+n).$\\

It is clear that if $z=4$ then $161+360({n}')=\rho_n$.\\

Also if ${n}'$ (and hence $z$) is increased by positive integers, then it is clear that by Dirichlet's theorem on arithmetic progressions $161 + 360({n}') = 161+360 (z+n)$ must produce infinite prime numbers for certain arbitrary values of $z$. Let the values of ${n}'$ be such that $161+360({n}')= \rho_{n}'$ is a prime number, are there arbitrary values for $z$ such that the prime numbers formed by ${n}'$ are of the form $k^2 +1$? indeed,\\

Let be $z(\alpha)=90\alpha^2+40\alpha-n+4$ ; with $\alpha=0,1,2,3,4,...$ With which you have to:
$k(\alpha)=180\alpha+40$ and $n(\alpha)=90\alpha^2+40\alpha$ for all $\alpha$ integer, since:
$$161+360(z+n)=\rho_{{n}'}=161+360(90\alpha^2+40\alpha-n+4+n)$$
$$=161+32400\alpha^2+14400\alpha+1440=1+(180\alpha+40)^2$$

$$\rho_{{n}'}=32400\alpha^2+14400\alpha+1601=1+(180\alpha+40)^2 ; \,\,\alpha=0,1,2,3,...$$

Then there is $\rho_n'>\rho_n$.\,\,\,$\rightarrow \leftarrow$
\end{proof}

The Table \ref{Table:vinculoinfinito} presents some of the results obtained that shows the formation of primes $\rho=k^2+1$ under the link function for this Ova-angular $\daleth_{161}$.\\

\begin{table}[ht]
\centering
\begin{tabular}[b]{|l||*{5}{c|}}\hline 
$\alpha$&$k(\alpha)=180\alpha+40$&$n(\alpha)=90\alpha^2+40\alpha$&$\Game_{\rho}$&$k^2+1$&Prime Number\\\hline \hline
$0$&$40$&$0$&$4$&$1601$&\checkmark\\\hline
$1$&$220$&$130$&$134$&$48401$&not\\\hline
$2$&$400$&$440$&$444$&$160001$&\checkmark\\\hline
$3$&$580$&$930$&$934$&$336401$&not\\\hline
$4$&$760$&$1600$&$1604$&$577601$&\checkmark\\\hline
$5$&$940$&$2450$&$2454$&$883601$&not\\\hline
$6$&$1120$&$3480$&$3484$&$1254401$&not\\\hline
$7$&$1300$&$4690$&$4694$&$1690001$&not\\\hline
$8$&$1480$&$6080$&$6084$&$2190401$&not\\\hline
$9$&$1660$&$7650$&$7654$&$2755601$&\checkmark\\\hline
$10$&$1840$&$9400$&$9404$&$3385601$&not\\\hline
$11$&$2020$&$11330$&$11334$&$4080401$&not\\\hline
$12$&$2200$&$13440$&$13444$&$4840001$&not\\\hline
$13$&$2380$&$15730$&$15734$&$5664401$&\checkmark\\\hline
$14$&$2560$&$18200$&$18204$&$6553601$&not\\\hline
\end{tabular}
\caption{Examples and link rotations between $ k(\alpha)$ and $n(\alpha)$ for $\daleth_{161}$ with $ n(\alpha)=\frac{k(\alpha)^2-1600}{360}.$}\label{Table:vinculoinfinito}
\end{table}

\textbf{Note $\textbf{4.}$} Baier and Zhao (2008) \cite{BAIERZHAO} present a summary of advances on primes in quadratic progressions and reference that Kuhn (1954) \cite {Kuhn} proved that $n^2 + 1$ is the product of three primes and if there are infinite primes in this way, they should be able to be expressed in an irreducible polynomial $an^2+bn+c$ with $a>0$ and $c$ odd number. 
To date, not all of these expressions have been found, but with the Ova-angular rotations the links functions between $k(\alpha)$ and $n(\alpha)$ are found, such that $\rho=k^2+1= an^2+bn+c$, with which we obtain all families of the progressions that Kuhn proposes that exist, as can be seen in Table \eqref{Table:vinculos}.
\begin{table}[ht]
\centering
\scalebox{0.64}{
\begin{tabular}[b]{|l||*{4}{c|}}\hline 
$\daleth_{n^2+1}$&$k(\alpha)$&$n(\alpha)$&$Link$&$a\alpha^2+b\alpha+c$\\\hline \hline
$1$&\begin{tabular}[c]{@{}c@{}}$A)\,180\alpha$\\ $B)\,180\alpha+120$\\$C)\,180\alpha-120$\end{tabular}&\begin{tabular}[c]{@{}c@{}}$90\alpha^2-40\alpha$\\ $90\alpha^2+120\alpha$\\$90\alpha^2-120\alpha$\end{tabular}&$n=\frac{k^2-120^2}{360}$&\begin{tabular}[c]{@{}c@{}}$32400\alpha^2+0\alpha+1$\\$32400\alpha^2+43200\alpha+14401$\\$32400\alpha^2-43200\alpha+14401$\end{tabular}\\\hline
$2$&$1$&$0$&$n=\frac{k^2-1^2}{360}$&$1^2+1=0\alpha^2+0\alpha+2$\\\hline
$5$&$2$&$0$&$n=\frac{k^2-2^2}{360}$&$2^2+1=0\alpha^2+0\alpha+5$\\\hline
$17$&\begin{tabular}[c]{@{}c@{}}$A)\,180\alpha+104$\\ $B)\,180\alpha+176$\\$C)\,180\alpha+4$\\$D)\,180\alpha-104$ \end{tabular}&\begin{tabular}[c]{@{}c@{}}$90\alpha^2+104\alpha+30$\\ $90\alpha^2+176\alpha+86$\\$90\alpha^2+4\alpha$\\$90\alpha^2-104\alpha+30$\end{tabular}&$n=\frac{k^2-4^2}{360}$&\begin{tabular}[c]{@{}c@{}}$32400\alpha^2+37440\alpha+10817$\\$32400\alpha^2+63360\alpha+30977$\\$32400\alpha^2+1440\alpha+17$\\$32400\alpha^2-37440\alpha+10817$\end{tabular}\\\hline
$37$&\begin{tabular}[c]{@{}c@{}}$A)\,360\alpha+54$\\ $B)\,360\alpha+126$\\$C)\,360\alpha+234$\\$D)\,360\alpha+306$\\$E)\,360\alpha+6$ \end{tabular}&\begin{tabular}[c]{@{}c@{}}$360\alpha^2+108\alpha+8$\\$360\alpha^2+252\alpha+44$\\$360\alpha^2+468\alpha+152$\\$360\alpha^2+612\alpha+260$\\$360\alpha^2+12\alpha$\end{tabular}&$n=\frac{k^2-6^2}{360}$&\begin{tabular}[c]{@{}c@{}}$129600\alpha^2+38880\alpha+2917$\\$129600\alpha^2+90720\alpha+15877$\\$129600\alpha^2+168480\alpha+54757$\\$129600\alpha^2+220320\alpha+93637$\\$129600\alpha^2+4320\alpha+37$\end{tabular}\\\hline
$41$&\begin{tabular}[c]{@{}c@{}}$A)\,180\alpha+160$\\ $B)\,180\alpha+20$ \end{tabular}&\begin{tabular}[c]{@{}c@{}}$90\alpha^2+160\alpha+70$\\$90\alpha^2+20\alpha$\end{tabular}&$n=\frac{k^2-20^2}{360}$&\begin{tabular}[c]{@{}c@{}}$32400\alpha^2+57600\alpha+25601$\\$32400\alpha^2+7200\alpha+401$\end{tabular}\\\hline
$77$&\begin{tabular}[c]{@{}c@{}}$A)\,180\alpha+34$\\ $B)\,180\alpha+74$\\$C)\,180\alpha+106$\\$D)\,180\alpha+146$ \end{tabular}&\begin{tabular}[c]{@{}c@{}}$90\alpha^2+34\alpha-12$\\ $90\alpha^2+74\alpha$\\$90\alpha^2+106\alpha+16$\\$90\alpha^2+146\alpha+44$\end{tabular}&$n=\frac{k^2-74^2}{360}$&\begin{tabular}[c]{@{}c@{}}$32400\alpha^2+12240\alpha+1157$\\$32400\alpha^2+26640\alpha+5477$\\$32400\alpha^2+38160\alpha+11237$\\$32400\alpha^2+52560\alpha+21317$\end{tabular}\\\hline
$101$&\begin{tabular}[c]{@{}c@{}}$A)\,180\alpha+10$\\ $B)\,180\alpha+170$ \end{tabular}&\begin{tabular}[c]{@{}c@{}}$90\alpha^2+10\alpha$\\$90\alpha^2+170\alpha+80$\end{tabular}&$n=\frac{k^2-10^2}{360}$&\begin{tabular}[c]{@{}c@{}}$32400\alpha^2+3600\alpha+101$\\$32400\alpha^2+61200\alpha+28901$\end{tabular}\\\hline
$137$&\begin{tabular}[c]{@{}c@{}}$A)\,180\alpha+64$\\ $B)\,180\alpha+136$\\$C)\,180\alpha-136$\\$D)\,180\alpha+116$ \end{tabular}&\begin{tabular}[c]{@{}c@{}}$90\alpha^2+64\alpha-26$\\ $90\alpha^2+136\alpha+14$\\$90\alpha^2-136\alpha+14$\\$90\alpha^2+116\alpha$\end{tabular}&$n=\frac{k^2-116^2}{360}$&\begin{tabular}[c]{@{}c@{}}$32400\alpha^2+23040\alpha+4097$\\$32400\alpha^2+48960\alpha+18497$\\$32400\alpha^2-48960\alpha+18497$\\$32400\alpha^2+41760\alpha+13457$\end{tabular}\\\hline
$161$&\begin{tabular}[c]{@{}c@{}}$A)\,180\alpha+40$\\ $B)\,180\alpha-40$ \end{tabular}&\begin{tabular}[c]{@{}c@{}}$90\alpha^2+40\alpha$\\$90\alpha^2-40\alpha$\end{tabular}&$n=\frac{k^2-40^2}{360}$&\begin{tabular}[c]{@{}c@{}}$32400\alpha^2+14400\alpha+1601$\\$32400\alpha^2-14400\alpha+1601$\end{tabular}\\\hline
$181$&\begin{tabular}[c]{@{}c@{}}$A)\,180\alpha+90$\\ $B)\,180\alpha+30$\\$C)\,180\alpha+150$\end{tabular}&\begin{tabular}[c]{@{}c@{}}$90\alpha^2+90\alpha$\\ $90\alpha^2+30\alpha-20$\\$90\alpha^2+150\alpha+40$\end{tabular}&$n=\frac{k^2-90^2}{360}$&\begin{tabular}[c]{@{}c@{}}$32400\alpha^2+32400\alpha+8101$\\$32400\alpha^2+10800\alpha+901$\\$32400\alpha^2+54000\alpha+22501$\end{tabular}\\\hline
$197$&\begin{tabular}[c]{@{}c@{}}$A)\,180\alpha+14$\\ $B)\,180\alpha+86$\\$C)\,180\alpha+94$\\$D)\,180\alpha+166$ \end{tabular}&\begin{tabular}[c]{@{}c@{}}$90\alpha^2+14\alpha$\\ $90\alpha^2+86\alpha+20$\\$90\alpha^2+94\alpha+24$\\$90\alpha^2+166\alpha+76$\end{tabular}&$n=\frac{k^2-14^2}{360}$&\begin{tabular}[c]{@{}c@{}}$32400\alpha^2+5040\alpha+197$\\$32400\alpha^2+30960\alpha+7397$\\$32400\alpha^2+33840\alpha+8837$\\$32400\alpha^2+59760\alpha+27557$\end{tabular}\\\hline
$217$&\begin{tabular}[c]{@{}c@{}}$A)\,360\alpha+144$\\ $B)\,360\alpha+216$\\$C)\,360\alpha+36$\\$D)\,360\alpha+324$\\$E)\,360\alpha+24$ \end{tabular}&\begin{tabular}[c]{@{}c@{}}$360\alpha^2+288\alpha+54$\\$360\alpha^2+432\alpha+126$\\$360\alpha^2+72\alpha$\\$360\alpha^2+648\alpha+288$\\$360\alpha^2+48\alpha-2$\end{tabular}&$n=\frac{k^2-36^2}{360}$&\begin{tabular}[c]{@{}c@{}}$129600\alpha^2+103680\alpha+20737$\\$129600\alpha^2+155520\alpha+46657$\\$129600\alpha^2+25920\alpha+1297$\\$129600\alpha^2+233280\alpha+104977$\\$129600\alpha^2+17280\alpha+577$\end{tabular}\\\hline
$221$&\begin{tabular}[c]{@{}c@{}}$A)\,180\alpha+70$\\ $B)\,180\alpha+110$ \end{tabular}&\begin{tabular}[c]{@{}c@{}}$90\alpha^2+70\alpha-20$\\$90\alpha^2+110\alpha$\end{tabular}&$n=\frac{k^2-110^2}{360}$&\begin{tabular}[c]{@{}c@{}}$32400\alpha^2+25200\alpha+4901$\\$32400\alpha^2+39600\alpha+12101$\end{tabular}\\\hline
$257$&\begin{tabular}[c]{@{}c@{}}$A)\,180\alpha+16$\\ $B)\,180\alpha+56$\\$C)\,180\alpha-56$\\$D)\,180\alpha-16$ \end{tabular}&\begin{tabular}[c]{@{}c@{}}$90\alpha^2+16\alpha$\\ $90\alpha^2+56\alpha+8$\\$90\alpha^2-56\alpha+8$\\$90\alpha^2-16\alpha$\end{tabular}&$n=\frac{k^2-16^2}{360}$&\begin{tabular}[c]{@{}c@{}}$32400\alpha^2+5760\alpha+257$\\$32400\alpha^2+20160\alpha+3137$\\$32400\alpha^2-20160\alpha+3137$\\$32400\alpha^2-5760\alpha+257$\end{tabular}\\\hline
$281$&\begin{tabular}[c]{@{}c@{}}$A)\,180\alpha+260$\\ $B)\,180\alpha-80$ \end{tabular}&\begin{tabular}[c]{@{}c@{}}$90\alpha^2+260\alpha$\\$90\alpha^2-80\alpha-170$\end{tabular}&$n=\frac{k^2-260^2}{360}$&\begin{tabular}[c]{@{}c@{}}$32400\alpha^2+93600\alpha+67601$\\$32400\alpha^2-28800\alpha+6401$\end{tabular}\\\hline
$317$&\begin{tabular}[c]{@{}c@{}}$A)\,180\alpha+26$\\ $B)\,180\alpha+46$\\$C)\,180\alpha+134$\\$D)\,180\alpha+154$ \end{tabular}&\begin{tabular}[c]{@{}c@{}}$90\alpha^2+26\alpha$\\ $90\alpha^2+46\alpha+4$\\$90\alpha^2+134\alpha+48$\\$90\alpha^2+154\alpha+64$\end{tabular}&$n=\frac{k^2-26^2}{360}$&\begin{tabular}[c]{@{}c@{}}$32400\alpha^2+9360\alpha+677$\\$32400\alpha^2+16560\alpha+2117$\\$32400\alpha^2+48240\alpha+17957$\\$32400\alpha^2+55440\alpha+23717$\end{tabular}\\\hline
$341$&\begin{tabular}[c]{@{}c@{}}$A)\,180\alpha+50$\\ $B)\,180\alpha+130$ \end{tabular}&\begin{tabular}[c]{@{}c@{}}$90\alpha^2+50\alpha-40$\\$90\alpha^2+130\alpha$\end{tabular}&$n=\frac{k^2-130^2}{360}$&\begin{tabular}[c]{@{}c@{}}$32400\alpha^2+18000\alpha+2501$\\$32400\alpha^2+46800\alpha+16901$\end{tabular}\\\hline \hline
\end{tabular}}
\caption{\,$a\alpha^2+b\alpha+c$ for all $\daleth_{n^2+1}$ with parameter $\alpha \in \mathbb{Z}$.}\label{Table:vinculos}
\end{table}

\clearpage

\begin{teor}[$\daleth_\rho$ Mersenne's Primes]\label{Conjetura Mersenne}
Let $\rho$ be a prime number and $M_\rho=2^{\rho}-1$ a Mersenne's prime. If $C_{Mersenne}$ is the set formed by the Ova-angular $\daleth$ of the Mersenne primes $2^\rho - 1$. Then:
$$C_{Mersenne}=\{ 3, 7, 31, 127, 247, 271\}.$$
\end{teor}

\begin{proof}
Let $M_\rho = 2^{\rho} -1$ be a Mersenne's prime, then there exists $\daleth \in C$ such that $M_\rho= \daleth + 360(\Game)$. Next we analyze which of the $99$ elements $\daleth \in C$ are possible to be Mersenne's primes, prior to this the following criteria will be taken into account: \\

\textit{Criterion A.} Given the expression $\daleth+360(\Game)=2^{\rho}-1$ it has to $\daleth+1=2^{\rho}-360(\Game)$, which indicates for $\rho>3$ that $\daleth+1$ must be divisible by $2,4$ and $8$.\\

\textit{Criterion B.} Given the expression $(\daleth+1)+360(\Game)=2^{\rho}$ it is possible in certain cases that $\daleth + 1$ is a multiple of $3 $ then it has the expression $3m+360(\Game)=2^{\rho}$. 
Taking common factor $3$ we have that $3k=2^{\rho}$ which is not possible since it contradicts the fundamental theorem of arithmetic and also $3$ does not divide any power of $2$. Consequently, this criterion affirms that $\daleth + 1$ cannot be a multiple of $3$.\\

\textit{Criterion C.} Given the expression $(\daleth+1)+360(\Game)=2^{\rho}$ it is possible in certain cases that $\daleth+1$  is a multiple of  $5$ then it has the expression $5n+360(\Game)=2^{\rho}$. Taking common factor $5$ we have that $5r=2^{\rho}$ which is not possible since $5r$ ends in $5$ or $0$ and the powers of $2$ end in $2,4,8,6$, that is; according to the fundamental theorem of arithmetic, $5$ does not divide any power of $2$. Consequently, this criterion states that $\daleth + 1$ cannot be a multiple of $5$.\\

\textit{Criterion D.} Let $\daleth \in \{103,223,343\}$ be, so for these values in the expression $\daleth+1=2^{\rho}-360(\Game)$ 
it is conditioned that $2^{\rho}$ must end in the digit $4$. Then $\rho \equiv 2 \pmod{4}$ this is $\rho=4m+2$, Then $\rho$ is even which would be contradictory since $\rho$ is odd prime. Consequently, this criterion affirms that $\daleth \neq  103,223,343.$\\

\textit{Criterio E.} Let $\daleth =151$ be, 
so for this value in equality $(\daleth+1)+360(\Game)=2^{\rho}$ 
you have infinite solutions to $\rho=3(4n-1)$ with $n \in \mathbb{N}$. 
Which is contradictory to the fact that $\rho$ must be prime. Consequently, this criterion affirms that $\daleth \neq 151.$\\

According to these criteria, you have to: \\

From \textit{Criterion A} the following are discarded $\{1,2,5,11,13,17,19,29,37,41,43,49,53,59,61,\\
67,73,77,83,89,91,97,101,107,109,113,121,131,133,137,139,149,157,161,163,167,169,\\
173,179,181,187,191,193,197,203,209,211,217,221,227,229,233,241,251,253,257,259,\\
263,269,277,281,288,289,293,299,301,307,313,317,323,329,331,337,341,347,349,353 \}.$ \\

From \textit{Criterion B} the following are discarded $\{23,47,71,119,143,287,311 \}.$\\

From \textit{Criterio C} the following are discarded $\{79,199,239,319,359 \}.$\\

From \textit{Criterio D} the following are discarded $\{103,223,343 \}.$\\

From \textit{Criterio E} it is discarded $\{151 \}.$\\

Then the only possible $\daleth \in C$ for a Mersenne's prime to be formed must be $\daleth \in \{3,7,31,127,247,271\}$.
\end{proof}

\textbf{Note $\textbf{5.}$} This is a necessary but not sufficient condition for a number to be from Mersenne. The Figure \eqref{Figura:Trapecio} shows the isosceles trapezoid that is formed by joining Mersenne's Ova-angular with an exponent greater than $3$. Hereafter this trapezoid will be called \textit {Mersenne's trapezoid Ova-angular}. Its symmetry is striking.

\begin{figure}[ht]
\centering
\includegraphics[width=1.97\textwidth]{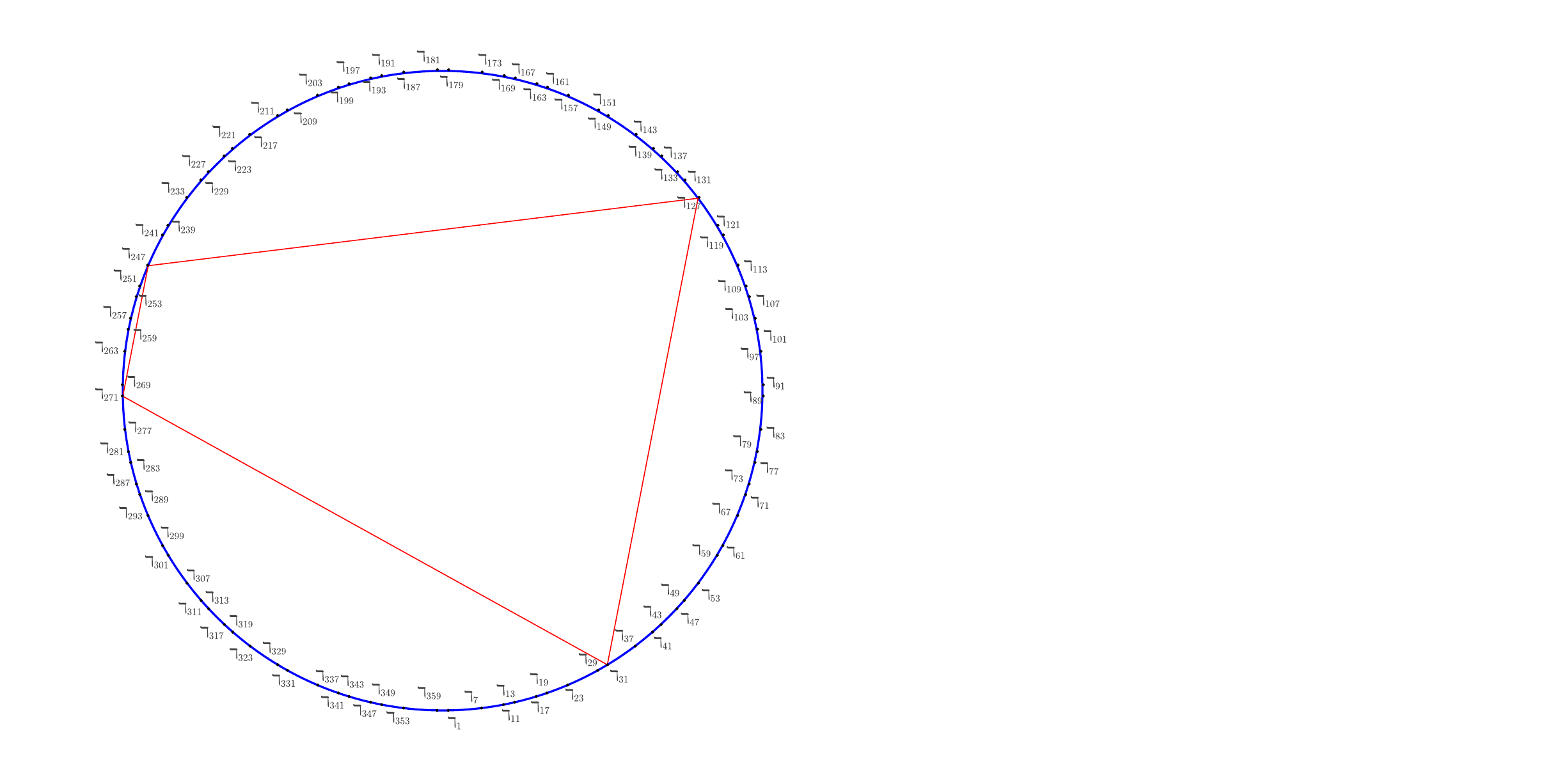}
\caption{Mersenne's trapezoid Ova-angular.}\label{Figura:Trapecio}
\end{figure}
\clearpage
The Table \eqref{Mersenne1} presents the Ova-angular residues for the $50$ Mersenne's primes found up to November 20, $2018$.\\
\begin{table}
\begin{center}
\begin{tabular}{|c|c|}\hline
\multicolumn{2}{|c|}{\vphantom{\LARGE Ap} OVA-ANGULAR RESIDUES FOR MERSENNE'S PRIMES $M_\rho$ }\\ \hline\hline
$M_2\equiv \daleth_{3}        \pmod{360}$ & $M_{23209}\equiv \daleth_{271}    \pmod{360}$\\
$M_3\equiv \daleth_{7}        \pmod{360}$ & $M_{44497}\equiv \daleth_{271}    \pmod{360}$\\
$M_5\equiv \daleth_{31}       \pmod{360}$ & $M_{86243}\equiv \daleth_{247}    \pmod{360}$\\
$M_7\equiv \daleth_{127}      \pmod{360}$ & $M_{110503}\equiv \daleth_{127}   \pmod{360}$\\
$M_{13}\equiv \daleth_{271}   \pmod{360}$ & $M_{132049}\equiv \daleth_{271}   \pmod{360}$\\
$M_{17}\equiv \daleth_{31}    \pmod{360}$ & $M_{216091}\equiv \daleth_{127}   \pmod{360}$\\
$M_{19}\equiv \daleth_{127}   \pmod{360}$ & $M_{756839}\equiv \daleth_{247}   \pmod{360}$\\
$M_{31}\equiv \daleth_{127}   \pmod{360}$ & $M_{859433}\equiv \daleth_{31}    \pmod{360}$\\
$M_{61}\equiv \daleth_{271}   \pmod{360}$ & $M_{1257787}\equiv \daleth_{127}  \pmod{360}$\\
$M_{89}\equiv \daleth_{31}    \pmod{360}$ & $M_{1398269}\equiv \daleth_{31}   \pmod{360}$\\
$M_{107}\equiv \daleth_{247}  \pmod{360}$ & $M_{2976221}\equiv \daleth_{31}   \pmod{360}$\\
$M_{127}\equiv \daleth_{127}  \pmod{360}$ & $M_{3021377}\equiv \daleth_{31}   \pmod{360}$\\
$M_{521}\equiv \daleth_{31}   \pmod{360}$ & $M_{6972593}\equiv \daleth_{31}   \pmod{360}$\\
$M_{607}\equiv \daleth_{127}  \pmod{360}$ & $M_{13466917}\equiv \daleth_{271} \pmod{360}$\\
$M_{1279}\equiv \daleth_{127} \pmod{360}$ & $M_{20996011}\equiv \daleth_{127} \pmod{360}$\\
$M_{2203}\equiv \daleth_{127} \pmod{360}$ & $M_{24036583}\equiv \daleth_{127} \pmod{360}$\\
$M_{2281}\equiv \daleth_{271} \pmod{360}$ & $M_{25964951}\equiv \daleth_{247} \pmod{360}$\\
$M_{3217}\equiv \daleth_{271} \pmod{360}$ & $M_{30402457}\equiv \daleth_{271} \pmod{360}$\\
$M_{4253}\equiv \daleth_{31}  \pmod{360}$ & $M_{32582657}\equiv \daleth_{31}  \pmod{360}$\\
$M_{4423}\equiv \daleth_{127} \pmod{360}$ & $M_{37156667}\equiv \daleth_{247} \pmod{360}$\\
$M_{9689}\equiv \daleth_{31}  \pmod{360}$ & $M_{42643801}\equiv \daleth_{271} \pmod{360}$\\
$M_{9941}\equiv \daleth_{31}  \pmod{360}$ & $M_{43112609}\equiv \daleth_{31}  \pmod{360}$\\
$M_{11213}\equiv \daleth_{31} \pmod{360}$ & $M_{57885161}\equiv \daleth_{31}  \pmod{360}$\\
$M_{19937}\equiv \daleth_{31} \pmod{360}$ & $M_{74207281}\equiv \daleth_{271} \pmod{360}$\\
$M_{21701}\equiv \daleth_{31} \pmod{360}$ & $\, M_{77232917}\equiv \, \daleth_{31} \pmod{360}$ $                                           $\\ \hline
\end{tabular}
\end{center}
\caption{Ova-angular residues for Mersenne`s primes $M_\rho$ (updated until November 2018).}\label{Mersenne1}
\end{table}

Similarly, the Table \eqref{Mersenne2} shows the Ova-angular residuals $\daleth_\rho$ for the primes $ \rho$ that generate $2^\rho-1$ Mersenne's prime. \\

\emph{Note:} From Mersenne's trapezoid and the two previous Tables,  it can be seen that there are disjoint subgroups for each $\daleth_\rho \in C_{Mersenne}$, in addition to interesting properties that are specified in the following theorems:

\begin{table}
\begin{center}
\begin{tabular}{|c|c|}\hline
\multicolumn{2}{|c|}{\vphantom{\LARGE Ap} OVA-ANGULAR RESIDUES FOR $\rho$ WITH $2^\rho-1$ MERSENNE'S PRIME}\\ \hline\hline
$2\equiv \daleth_{2}       \pmod{360}$ & $23209\equiv \daleth_{169} \pmod{360}$\\
$3\equiv \daleth_{3}       \pmod{360}$ & $44497\equiv \daleth_{217} \pmod{360}$\\
$5\equiv \daleth_{5}       \pmod{360}$ & $86243\equiv \daleth_{203} \pmod{360}$\\
$7\equiv \daleth_{7}       \pmod{360}$ & $110503\equiv \daleth_{343} \pmod{360}$\\
$13\equiv \daleth_{13}     \pmod{360}$ & $132049\equiv \daleth_{289} \pmod{360}$\\
$17\equiv \daleth_{17}     \pmod{360}$ & $216091\equiv \daleth_{91} \pmod{360}$\\
$19\equiv \daleth_{19}     \pmod{360}$ & $756839\equiv \daleth_{119} \pmod{360}$\\
$31\equiv \daleth_{31}     \pmod{360}$ & $859433\equiv \daleth_{113} \pmod{360}$\\
$61\equiv \daleth_{61}     \pmod{360}$ & $1257787\equiv \daleth_{307} \pmod{360}$\\
$89\equiv \daleth_{89}     \pmod{360}$ & $1398269\equiv \daleth_{29} \pmod{360}$\\
$107\equiv \daleth_{107}   \pmod{360}$ & ${2976221}\equiv \daleth_{101} \pmod{360}$\\
$127\equiv \daleth_{127}   \pmod{360}$ & ${3021377}\equiv \daleth_{257} \pmod{360}$\\
$521\equiv \daleth_{161}   \pmod{360}$ & ${6972593}\equiv \daleth_{113} \pmod{360}$\\
$607\equiv \daleth_{247}   \pmod{360}$ & ${13466917}\equiv \daleth_{37} \pmod{360}$\\
$1279\equiv \daleth_{199}  \pmod{360}$ & $20996011\equiv \daleth_{91} \pmod{360}$\\
$2203\equiv \daleth_{43}   \pmod{360}$ & $24036583\equiv \daleth_{103} \pmod{360}$\\
$2281\equiv \daleth_{121}  \pmod{360}$ & $25964951\equiv \daleth_{311} \pmod{360}$\\
$3217\equiv \daleth_{337}  \pmod{360}$ & $30402457\equiv \daleth_{97} \pmod{360}$\\
$4253\equiv \daleth_{293}  \pmod{360}$ & $32582657\equiv \daleth_{137} \pmod{360}$\\
$4423\equiv \daleth_{103}  \pmod{360}$ & $37156667\equiv \daleth_{347} \pmod{360}$\\
$9689\equiv \daleth_{329}  \pmod{360}$ & $42643801\equiv \daleth_{1} \pmod{360}$\\
$9941\equiv \daleth_{221}  \pmod{360}$ & $43112609\equiv \daleth_{89} \pmod{360}$\\
$11213\equiv \daleth_{53}  \pmod{360}$ & $57885161\equiv \daleth_{41} \pmod{360}$\\
$19937\equiv \daleth_{137} \pmod{360}$ & $74207281\equiv \daleth_{121} \pmod{360}$\\
$21701\equiv \daleth_{101} \pmod{360}$ & $77232917\equiv \daleth_{317} \pmod{360}$ \\ \hline
\end{tabular}
\end{center}
\caption{Ova-angular residues for $\rho$ prime number and exponent in $2^\rho-1$ Mersenne's prime.}\label{Mersenne2}
\end{table}

\begin{teor}\label{T14}
All Mersenne prime number $M_\rho$ is congruent to $3$ Module $4$. That is, all Mersenne prime  is of the form $4k-1$, $k \in \mathbb{Z}$.
\end{teor}

\begin{proof}
Let $ M_\rho $ be a Mersenne's prime. So $M_{\rho} \equiv \daleth_{\rho} \pmod{360}$. Since $360=2^3 \,3^2 \,5$, by the "Chinese remainder theorem" it is possible to affirm that $M_\rho \equiv \daleth_\rho \pmod{2^3}$.
Since the Ova-angular residue for every Mersenne prime is $\daleth_{\rho} \in C_{Mersenne}=\{3, 7,31,127,247,271\}$ then $M_{\rho} \equiv \daleth_{\rho} \pmod{8}$, so $2M_{\rho} \equiv 2 \daleth_{\rho} \pmod{8}$.\\

Then as $ (2,8) = 2 $ then for properties of the congruences we have to $M_\rho \equiv \daleth_{\rho} \pmod{4}.$ Furthermore, it is true that $\daleth_{\rho} \equiv 3 \pmod{4}$ for all $\daleth_\rho \in C_{Mersenne}$. Then
$$M_{\rho} \equiv 3 \pmod{4}.$$ \end{proof}

\begin{corol}
Every Mersenne prime is a Gaussian prime number, that is, every Mersenne prime remains inert in $\mathbb{Z}[i].$
\end{corol}

\begin{proof}
As every Mersenne prime is of the form $4k-1$ then when making the change of variable $k= t + 1 $ we have that it is of the form $4t+3$. 
\end{proof}

\begin{teor}
Let $M_\rho$ be a Mersenne's prime and $\rho>2$. Then, every Mersenne prime $M_\rho$ is congruent
with $7\pmod{24}$, with $7\pmod{12}$, with $7\pmod{8} $, with $1\pmod{6}$, with $3\pmod{4}$, with $1\pmod{3}$ and with $1\pmod{2}$. That is, all Mersenne prime with prime exponent greater than $2$ is of the form $24z+7$, $12r+7$, $8t+7$, $6s+1$, $4q+3$, $3n+1$, $2m+1$ y $4k-1$, with $s,n,m,k \in \mathbb{Z}$.
\end{teor}

\begin{proof}
Process similar to proof of Theorem \eqref{T14}.
\end{proof}

\begin{teor}[Singular Mersenne ]
The number of Mersenne primes in class $\daleth_{3}$ and in class $\daleth_{7}$ equals $1$, i.e, the only primes of Mersenne and that from now on will be called singular Mersenne are $2^2-1$ and $2^3-1$ in their respective residual class.
\end{teor}

\begin{proof}
For the $\daleth_{3}$ of the Mersenne primes 
you get that $3+360k=2^{\rho}-1$ for some $k$ integer.
Suppose that there is another prime $M_{\rho} \neq 2^2-1$ with the same residue  and $\rho>2$. Then this prime must satisfy for $k$ integer:
$$k=\frac{2^{\rho}-4}{360}=\frac{2^{(\rho-2)}-1}{90} \rightarrow \,\,\,k \,\textit{is a pure fractional.}$$
\textbf{Note $\textbf{6.}$} $ k $ is a pure fractional since $2^{n}-1 $ ends in $1,3,5,7$, so it would never be a multiple of $ 90$. \\

Then $k$ is an integer and $k$ is a pure fractional $\rightarrow \leftarrow$. Then the only Mersenne's prime in the class $\daleth_3$ is $2^{2}-1=3$.\\

Now, for the $\daleth_{7}$ class of Mersenne primes, we have to: $7+360k=2^{\rho}-1$. Suppose that there is another prime number $M_{\rho} \neq 2^3-1$ with the same residue and $\rho>3$. Then this prime must satisfy for $k$ integer:
$$k=\frac{2^{\rho}-8}{360}=\frac{2^{(\rho-8)}-1}{45},$$

Since $k$ is a integer then $2^{\rho-8}-1$ it must end in $5$ or $0$, but since it never ends in $0$ then the only option is that it must end in $5$, so $2^{\rho-8}$ must end in $6$. Then $\rho-8 \equiv 0 \pmod{4} \Rightarrow \rho \equiv 0 \pmod{4}$, then $\rho$ is 
an even prime number$\rightarrow \leftarrow$. Then the only Mersenne's prime in the class $\daleth_{7}$ es $2^{3}-1=7$.
\end{proof}

\begin{teor}\label{16T}
Let $n \in  \mathbb{N}$ be, then Mersenne's primes in the class $\daleth_{31}$ are in the sequence:
$$31+360K_n=2^{\rho_n}-1,\,\,\,\,\text{where}\,\,K_n=\frac{2^{12n-10}-2^2}{3^2 5} \,\,\,\,\,\,\,\, \text{and} \,\,\,\,\,\,\,\, \rho_n=12n-7=12(n-1)+5.$$
\end{teor}

\begin{proof}
Since $2^{\rho}-1=31+360(k) $ then we have that $\rho = \log_{2} {(32 + 360k)}$, now since $\rho$ is a natural number then it is convenient analyze the values of $k$ for which $\log_{2}{(32 + 360k)}$ intersects with some natural. 
Let it be
$$k=K_n=\frac{2^{12n-10}-2^2}{3^2. 5},$$ then it is clear that
$$\rho_n= \log_{2}{(32+360K_n)}=\log_{2}{\left (32+360\frac{2^{12n-10}-2^2}{3^2 .5}\right )}=12n-7,$$
sequence in which we have the only integer solutions of the required logarithm.
\end{proof}

\textbf{Note $\textbf{7.}$} It's clear that if $31+360K_n$ is prime then it is also a Mersenne's prime, and 
according to Dirichlet's theorem, if there are infinite values of $K_n$ that make prime numbers, then there are infinite Mersenne primes.

\begin{corol} \label{Col7}
For the class $\daleth_{31}$ of Mersenne's primes, the only $\daleth_\rho$ for the prime exponent $\rho$ are:
$5, 17, 29, 41, 53, 77, 89, 101, 113, 137, 149, 161, 173, 197, 209, 221, 233, 257, 269, 281, \\
293, 317, 329, 341, 353.$
\end{corol}

\begin{proof}
Since $\rho_n=12n-7$ then $\daleth+360\Game=12n-7$, also $(\daleth+7)=12(n-30\Game)$. This conditions that $12 | (\daleth+7)$ and so the only possible Ova-angular $\daleth \in C$ for the exponent are:\\
$\{5, 17, 29, 41, 53, 77, 89, 101, 113, 137, 149, 161, 173, 197, 209, 221, 233, 257,
269, 281, 293, 317,\\
329, 341, 353\}.$
\end{proof}

\begin{teor}
Let $m \in \mathbb{N}$ be, then Mersenne's primes in the class $\daleth_{127}$ are in the sequence:
$$127+360K_m=2^{\rho_m}-1,\,\,\,\,\text{where}\,\,\,K_m=\frac{2^{12m-8}-2^4}{3^2 5} \,\,\,\,\,\,\,\, \text{and} \,\,\,\,\,\,\,\,\rho_m=12m-5=12(m-1)+7.$$
\end{teor}

\begin{proof}
Similar to the proof of the previous theorem \eqref{16T}.
\end{proof}

\begin{corol}
For the class $\daleth_{127}$ of Mersenne's primes, the only $\daleth_\rho$ for the prime exponent $\rho$ are: $7, 19, 31, 43, 67, 79, 91, 103, 127, 139, 151, 163, 187, 199, 211,223, 247, 259, 271, 283, 307,\\ 319, 331, 343.$
\end{corol}

\begin{proof}
Analogous to the previous Corollary \eqref{Col7} proof,  only that $ 12 | (\daleth + 5)$.
\end{proof}

\textbf{Note $\textbf{8.}$} It's clear that if $127+360K_m$ is prime then it is also a Mersenne's prime, and 
according to Dirichlet's theorem, if there are infinite values of $K_m$ that make prime numbers, then there are infinite Mersenne primes.

\begin{teor}
Let $r \in \mathbb{N}$ be, the Mersenne's primes in the class $\daleth_{247}$ are in the sequence:
$$247+360K_r=2^{\rho_r}-1,\,\,\,\,\text{where}\,\,\,K_r=\frac{2^{12r-4}-31}{3^2 5} \,\,\,\,\,\,\,\, \text{and} \,\,\,\,\,\,\,\, \rho_r=12r-1=12(r-1)+11.$$
\end{teor}

\begin{proof}
Similar to the proof of the previous theorem \eqref{16T}.
\end{proof}

\begin{corol}
For the class $\daleth_{247}$ of Mersenne's primes, the only $\daleth_\rho$ for the prime exponent $\rho$ are: $11, 23, 47, 59, 71, 83, 107, 119, 131, 143, 167, 179, 191, 203, 227, 239, 251,263, 287, 299,\\ 311, 323, 347, 359.$
\end{corol}

\begin{proof}
Analogous to the previous Corollary \eqref{Col7} proof,  only that $ 12 | (\daleth + 1)$.
\end{proof}

\textbf{Note $\textbf{9.}$}
It's clear that if $247+360K_r$ is prime then it is also a Mersenne's prime, and according to Dirichlet's theorem, if there are infinite values of $K_r$ that make prime numbers, then there are infinite Mersenne primes.

\begin{teor}
Let $s \in \mathbb{N}$ be, then Mersenne's primes in the class $\daleth_{271}$ are in the sequence:
$$271+360K_s=2^{\rho_s}-1,\,\,\,\,\text{where}\,\,\,K_s=\frac{2^{12s-14}-34}{3^2 5} \,\,\,\,\,\,\,\, \text{and} \,\,\,\,\,\,\,\, \rho_s=12s-11=12(s-1)+1.$$
\end{teor}

\begin{teor}
For the class $\daleth_{271}$ of Mersenne's primes, the only $\daleth_\rho$ for the prime exponent $\rho$ are:
$1, 13, 37, 49, 61, 73, 97, 109, 121, 133, 157, 169, 181, 193, 217,229, 241, 253, 277, 289, 301,\\
313, 337, 349.$
\end{teor}

\begin{proof}
Analogous to the previous Corollary \eqref{Col7} proof,  only that $ 12 | (\daleth + 11)$.
\end{proof}

\textbf{Note $\textbf{10.}$}
It's clear that if $271+360K_s$ is prime then it is also a Mersenne's prime, and 
according to Dirichlet's theorem, if there are infinite values of $K_s$ that make prime numbers, then there are infinite Mersenne primes.\\

As previously mentioned, all established subgroups $\daleth$ for each Mersenne's prime family $c_{Mersenne}$ are disjoint.
\clearpage

\begin{teor}[Mersenne's primes inverse series] The series formed by the inverses of the Mersenne's primes converge.
\end{teor}

\begin{proof}
The theorem is equivalent to proving that
$$\lim_{x \to \infty}  \left ( \sum _{M_{\rho}\leqslant x} \frac{1}{M_\rho}\right )<\infty \,\,\,\,\,\,\,\,;\,\,\,\,\text{with}\,\,M_{\rho}\,\,\in\mathbb{P}.$$

Let's analyze the series given by:
$$\sum _{M_{\rho}} \frac{1}{M_\rho}=$$

$ \left ( \frac{1}{M_2} \right )+\left ( \frac{1}{M_3} \right )+\left ( \frac{1}{M_5}+\frac{1}{M_{17}}+\frac{1}{M_{1398269}} +...+\frac{1}{M_{2^{(12k-7)}}}+...\right )+\left ( \frac{1}{M_7}+\frac{1}{M_{19}}+\frac{1}{M_{31}} +...+\frac{1}{M_{2^{(12t-5)}}}\right )+\left ( \frac{1}{M_{107}}+\frac{1}{M_{86243}}+\frac{1}{M_{756839}} +...+\frac{1}{M_{2^{(12w-1)}}}\right ) +\left ( \frac{1}{M_{13}}+\frac{1}{M_{61}}+\frac{1}{M_{2281}} +...+\frac{1}{M_{2^{(12f-11)}}}\right ),$\\

then grouping into subgroups $\daleth_3,\daleth_7,\daleth_{31},\daleth_{127},\daleth_{247},\daleth_{271}$ we obtain
$$\sum _{M_{\rho}} \frac{1}{M_\rho}= \left ( \frac{1}{M_2} \right )+\left ( \frac{1}{M_3} \right )+\left ( \sum_{i=n}^{\infty}{\frac{1}{2^{(12i-7)}-1}} \right )_{\daleth_{31}}+\left ( \sum_{i=m}^{\infty}{\frac{1}{2^{(12i-5)}-1}} \right )_{\daleth_{127}}+$$

$$\left ( \sum_{j=r}^{\infty}{\frac{1}{2^{(12j-1)}-1}} \right )_{\daleth_{247}}+
\left ( \sum_{j=s}^{\infty}{\frac{1}{2^{(12j-5)}-1}} \right )_{\daleth_{271}}.$$

These last expressions, according to the criterion of Alembert's relation, converge, i.e:
$$\sum _{M_{\rho}} \frac{1}{M_\rho}= \left ( \frac{1}{M_2} \right )+\left ( \frac{1}{M_3} \right )+C_1+C_2+C_3+C_4=C.$$

$$\therefore \,\, \sum _{M_{\rho}} \frac{1}{M_\rho}<\infty  \,\,\,\,\,\,\,\,\,\, \Rightarrow \,\,\,\,\,\,\, \lim_{x \to \infty}  \left ( \sum _{M_{\rho}\leqslant x} \frac{1}{M_\rho}\right )<\infty . $$
\end{proof}

\begin{defin}[Mersenne's constan]
Let $C_{M_\rho}$ be the value el valor to which the sum of the inverses of the Mersenne's primes converges, then:

$$C_{M_\rho}= \left ( \frac{1}{M_2} \right )+\left ( \frac{1}{M_3} \right )+\left ( \frac{1}{M_5} \right )+\left ( \frac{1}{M_7} \right )+...$$

$$C_{M_\rho}=0.516454178940788565330487342971522858815968553415419701441.$$
\end{defin}
 
\begin{teor}
For every Mersenne's prime $M_{\rho}>3$, it is true that $M_{\rho}+2$ is not a prime number. That is, there is not Mersenne's prime $M_\rho>3$ that has a twin prime.
\end{teor}

\begin{proof}
According to the Theorem \eqref{Conjetura Mersenne} the only ones possible $\daleth$ for a Mersenne's prime are $\daleth_{31},\,\daleth_{127},\,\daleth_{247},\,\daleth{271}$ and by adding $2$ we observe that there is not possible Ova-angular $\daleth \in C$, since its next Ova-angular are separated by a greater distance, therefore there cannot be any prime of the form $M_{\rho}+2$.
\end{proof}

\begin{teor}
Let $C_{Germain}$ be the set of all Ova-angular $\daleth_\rho$ such that $\rho=2\rho_i + 1$ is prime, i.e. the set of all Ova of the primes formed with the Germain's primes. Then:\\

$C_{Germanin}=\{ 5, 7, 11, 23, 23, 47, 59, 83, 107, 119, 143, 167, 179, 187, 191, 203, 227, 239, 263,\\ 
287, 299, 323, 347, 359\}$
\end{teor}

\begin{proof}
The proof is similar to the Ova-angular residues found for  Mersenne's primes.
\end{proof}

The Figure \eqref{Figura:OvaGermain} shows the connection of all residuals $\daleth\neq5$ of the primes generated with the Germain's primes. Its reflective symmetry is striking.
\begin{figure}[ht]
\centering
\includegraphics[width=1.569\textwidth]{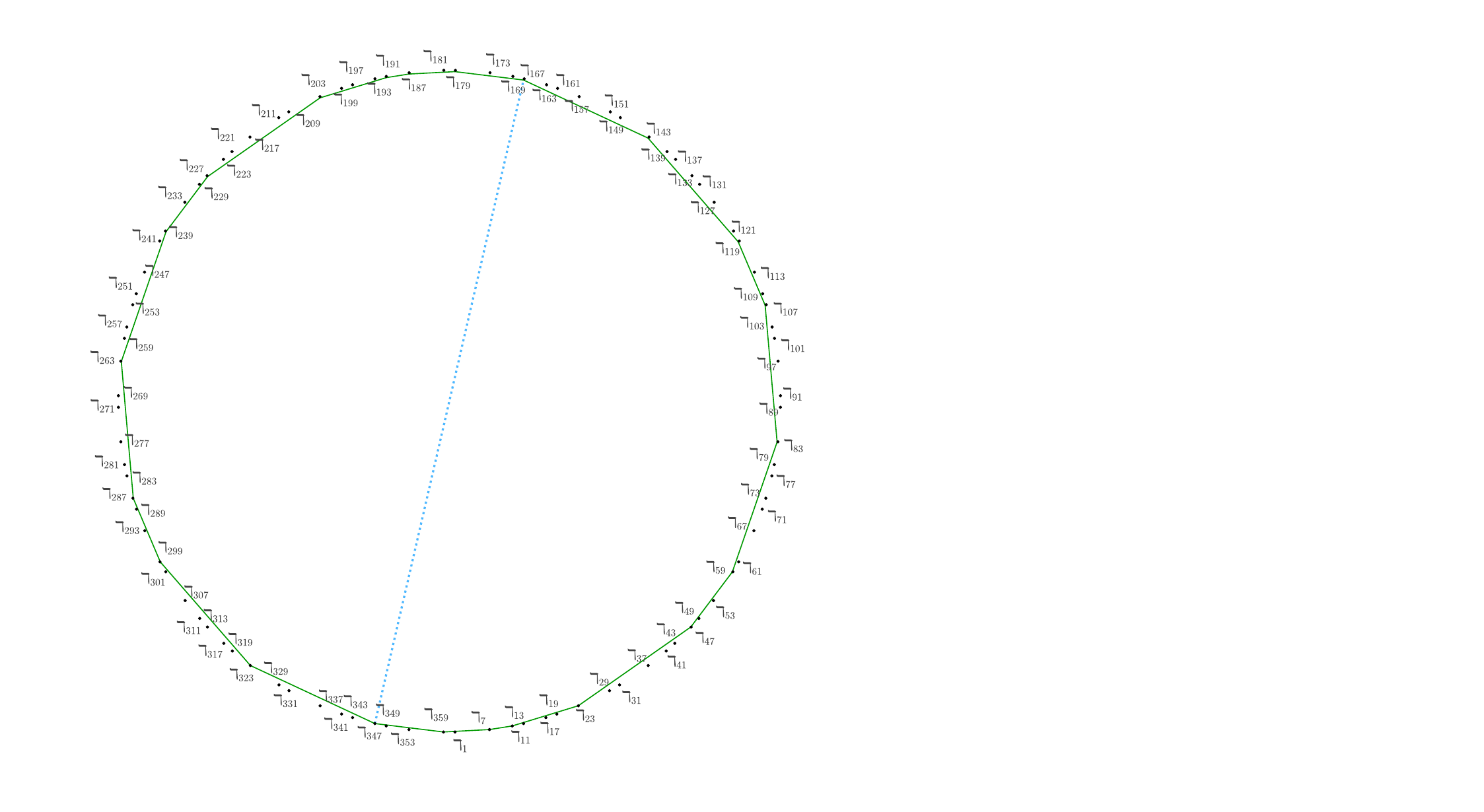}
\caption{Ova -$\daleth_{Germain}$ conection.}\label{Figura:OvaGermain}
\end{figure}


\begin{prop}[Equivalence with the Prime Number Theorem]
$$\lim_{x \rightarrow \infty} \frac{\pi(x)}{x\,/\, log(x)}=1\,\,\,\,\,\, \Rightarrow\,\,\,\,\,\,\lim_{x \rightarrow \infty} \frac{\pi(x,\daleth_\rho,360) \phi(360)}{x\,/\, log(x)}=1.$$ 
\end{prop}

\begin{proof}
This proposition is already proven, read about it in \cite{Breitzman1970} and \cite{Goldston-Pintz-Yildirim}, and other recent details in \cite{GOLDSTON2020}.
\end{proof}

Some final properties of the Ova-angular rotations are listed below, the proofs are omitted since they are immediate.

\begin{teor}[Sum of digits of a prime]
Let $S_\rho$ be the sum of the digits of $\rho$ prime number, then $S_\rho \equiv {\daleth_\rho} \pmod{9}$, and $\,S_\rho \equiv {\daleth_\rho} \pmod{3}$.
\end{teor}

\begin{teor}[Necessary but not sufficient condition]
If $\rho=\daleth_\rho + 360(\Game_\rho)$, with $\rho$ prime number and $\Game_\rho >0$ then $gcd (\daleth_\rho, \Game_\rho)=1$, and  $lcm(\daleth_\rho,\Game_\rho)=\daleth_\rho . \Game_\rho$.
\end{teor}

\begin{teor}
\label{wilson interpretacion}
Let $n \in \mathbb{N}$, $n$ is a prime number 
if and only if $\frac{(n)!+n}{n^2}$ $\pmod{360}= k$; for some integer $k<360$.
\end{teor}
\begin{teor}
Let $\rho$ and $\rho+2$ prime numbers. If $\daleth_\rho \neq 359 $ then the frequency of rotation of both prime numbers is the same.
\end{teor}

\begin{teor}
Neither Germain's nor Mersenne's primes are of the form $k^2+1$, for some integer $k$.
\end{teor}

\begin{teor}
There cannot be a Mersenne's prime who is in turn Germain's prime.
\end{teor}

\begin{teor}
Let $\rho>5$ be a prime number. If $\rho+2$ is a prime number then $\rho-2$ is a compound number, i.e. there is not trilogy of consecutive primes greater than $3,5,7$.
\end{teor}

\begin{teor}
Let $M_\rho$ be a Mersenne's prime then $M_\rho+6n-4 \neq 3$ is always a composite number, for everything $n \in \mathbb{Z}$.
\end{teor}

\begin{teor}
Let $M_\rho$ be a Mersenne's prime. If $M_\rho \pmod{360}= \daleth_{127}$ or $\daleth_{247}$ then $M_\rho -2$ is a composite number.
\end{teor}

\begin{teor}
Let $M_\rho$ be a Mersenne's prime. If $M_\rho \pmod{360}= \daleth_{31}$ or $\daleth_{271}$ then $M_\rho +4$ is a composite number.
\end{teor}

\clearpage


\begin{defin}[Ova-angular matrix]

Let $0 < \Game_x $ be an integer, the frequency function is defined 
\[f_{\daleth{\rho}}(\Game_x)=
\begin{cases}
1 & \text{si $\, \daleth_{\rho}+360\Game_x\,$ es número primo},\\
0 & \text{en cualquier otro caso.}
\end{cases} \]
and its representative square matrix determined by $x_{i,j_k}=f(\Game_{k(i-1)+j}).$ 
\end{defin}

\textbf{Example:} The squeare matrix ($k=10$) of $100$ Ova-angular rotations for $\rho_7$ and $\rho_{353}$ respectively is presented below.
\begin{center}
\begin{tabular}{c|c}
Ova-angular matrix $\rho_7.$ & Ova-angular matrix $\rho_{353}.$ \\
{$\begin{bmatrix}
1 & 1 & 1 & 1 & 0 & 0 & 0 & 1 & 0 & 1\\
1 & 1 & 0 & 0 & 1 & 0 & 0 & 0 & 0 & 1\\
0 & 1 & 1 & 1 & 1 & 0 & 0 & 0 & 0 & 0\\
0 & 1 & 1 & 0 & 0 & 1 & 1 & 1 & 0 & 1\\
1 & 0 & 0 & 0 & 0 & 1 & 1 & 0 & 0 & 0\\
1 & 0 & 1 & 1 & 0 & 0 & 0 & 1 & 1 & 0\\
0 & 0 & 0 & 0 & 0 & 1 & 0 & 0 & 1 & 0\\
0 & 0 & 0 & 1 & 0 & 1 & 0 & 1 & 1 & 1\\
1 & 1 & 0 & 0 & 0 & 0 & 1 & 1 & 0 & 0\\
0 & 0 & 1 & 0 & 0 & 0 & 0 & 0 & 0 & 1
\end{bmatrix}$}
& {$\begin{bmatrix}
0 & 0 & 1 & 0 & 1 & 0 & 0 & 0 & 1 & 0\\
0 & 1 & 0 & 1 & 0 & 1 & 1 & 1 & 1 & 0\\
0 & 1 & 0 & 0 & 0 & 0 & 0 & 1 & 0 & 0\\
0 & 0 & 0 & 0 & 1 & 1 & 0 & 1 & 0 & 1\\
0 & 1 & 0 & 1 & 1 & 0 & 0 & 0 & 0 & 1\\
1 & 1 & 1 & 1 & 0 & 0 & 1 & 0 & 0 & 0\\
0 & 0 & 0 & 0 & 1 & 1 & 1 & 0 & 0 & 0\\
1 & 0 & 1 & 1 & 0 & 0 & 0 & 1 & 1 & 1\\
0 & 1 & 0 & 1 & 0 & 0 & 0 & 0 & 0 & 0\\
1 & 0 & 0 & 0 & 0 & 1 & 0 & 0 & 1 & 1
\end{bmatrix}$} \\
\end{tabular}
\end{center}

It is clear that this matrix representation presents the prime numbers for $\rho_7$ and $\rho_{353}$ in an interval $0<\Game_{n}\leq100$ rotations, 
so for example for the right matrix, it happens that
in row $1$ and column $5$ we have  $f_{\daleth_{353}}(5)=1)$ with $\Game_{x}=5$, this indicates that the number $353+360(5)=2153$ is a prime number. The zeros indicate that a prime number is not formed in these rotations.\\

We can notice that this type of matrix representation for the prime numbers according to their Ova-angular rotations overall presents a better way for its study and makes it an object for statistical analysis and linear algebra, which will be studied in part $II$, in particular, the square matrices $360*360$ rotations.


\section{Conclusions}
In this work, the Ova-angular rotations of a prime number were characterized. Although the mathematics that was used was elementary, you can notice the usefulness of this theory based on geometric properties. Finally, some important applications were presented and in the case of Mersenne's primes, another way was provided for their study and search in disjoint subgroups. It is possible to continue this work analyzing other applications that this theory presents using the geometric properties for different prime numbers, the frequency function, and the Ova-angular matrices. \\

\section*{Referencias}
\bibliography{reference.bib}

\end{document}